\documentclass[]{interact}

\usepackage{epstopdf}
\usepackage[caption=false]{subfig}
\usepackage{placeins} 
\usepackage{float} 
\usepackage[numbers,sort&compress]{natbib}
\bibpunct[, ]{[}{]}{,}{n}{,}{,}
\renewcommand\bibfont{\fontsize{10}{12}\selectfont}
\makeatletter
\def\NAT@def@citea{\def\@citea{\NAT@separator}}
\makeatother
\usepackage{mathrsfs}
\theoremstyle{plain}
\newtheorem{theorem}{Theorem}[section]
\newtheorem{lemma}[theorem]{Lemma}
\newtheorem{corollary}[theorem]{Corollary}
\newtheorem{proposition}[theorem]{Proposition}

\theoremstyle{definition}
\newtheorem{definition}[theorem]{Definition}

\theoremstyle{remark}
\newtheorem{remark}{Remark}

\DeclareMathOperator{\supp}{supp}
\DeclareMathOperator{\soft}{soft}
\DeclareMathOperator{\co}{co}
\newcommand{\R}{\mathbb R}
\newcommand{\norm}[1]{\left\|#1\right\|}
\newcommand{\ip}[2]{\left\langle #1,#2\right\rangle}
\newcommand{\eps}{\varepsilon}
\begin{document}


\title{Explicit Criteria for Finite-Time Identification
in Sparse Primal--Dual Dynamics}

\author{
\name{H. Saoud \thanks{CONTACT H. Saoud Email: saoud.h@gust.edu.kw}}
\affil{Department of Mathematics and Natural Sciences, Gulf University for Science and Technology, P.O. Box 7207, Hawally 32093, Kuwait.}
}

\maketitle

\begin{abstract}
We study when selected coordinates become and remain zero
in a finite-dimensional sparse-regression primal--dual flow
and its exact full resolvent iteration. Invariant balls
centered at \((0,z)\), with \(z\) strictly dual feasible,
give computable deadlines without rank assumptions or
optimizer knowledge. The certificates may be available
at initialization. We determine the smallest invariant
radius for the continuous flow within this family.
For a prescribed positive feasibility margin, center
selection is convex and admits an exact finite conic
formulation. We bound the coordinate margins for all
centers in this family and give a normalized correlated
example attaining the bound. A composite Bregman Lyapunov
estimate proves rank-free convergence without primal
uniqueness and finite identification under strict dual
margins. Under restricted injectivity, the discrete
gap-based deadlines multiplied by the step size tend
to the continuous bound for fixed data, observed state,
and certificate parameters. Numerical experiments assess
certificate timing. A noisy weak-form oscillator regression
illustrates the difference between fitted and physical support.
\end{abstract}

\begin{amscode}
Primary 90C25; Secondary 34A60, 47H05, 49J52, 65K05.
\end{amscode}

\section{Introduction}\label{sec:introduction}

The \emph{least absolute shrinkage and selection operator}
(LASSO), introduced by Tibshirani in 1996 \cite{Tibshirani},
adds an \(\ell^1\) penalty to a least-squares objective.
This penalty shrinks the fitted coefficients and can set some
exactly to zero. It therefore selects features from a prescribed
dictionary. One application is the estimation of governing
equations from measurements, using candidate nonlinear terms
\cite{Brunton}. A weak formulation builds the regression by
integrating against test functions. This avoids pointwise
differentiation of noisy measurements \cite{Messenger}.
Once the dictionary and observations are fixed, the regression
problem is finite dimensional. We study the optimization dynamics
used to solve it and, in particular, when their coordinates
become and remain zero. With noisy data, a zero fitted coefficient
need not represent a term absent from the physical system.

Let \(m,p\ge1\) be integers, \(B\in\R^{m\times p}\),
\(b\in\R^m\), and \(\lambda,\gamma>0\). With
\(g(x)=\lambda\norm{x}_1\), we consider the primal--dual
differential inclusion
\begin{equation}\label{flow}
\dot x(t)+\partial g(x(t))\ni-B^\top y(t),
\qquad
\dot y(t)=Bx(t)-b-\gamma y(t),
\end{equation}
where \(\partial g\) is the convex subdifferential of \(g\).
This is the saddle evolution associated with
\(\mathcal L(x,y)=g(x)+\ip{y}{Bx-b}-\gamma\norm{y}^2/2\).
At equilibrium, the primal component minimizes
\begin{equation}\label{primal}
P(x)=\lambda\norm{x}_1+\frac{1}{2\gamma}\norm{Bx-b}^2,
\qquad x\in\R^p.
\end{equation}
Each \(x_i\) is the coefficient of the feature \(B_i\).
The equilibrium relation \(y=(Bx-b)/\gamma\) identifies
the dual variable as a scaled residual. Time in \eqref{flow}
measures the progress of optimization. It is distinct from
the physical time used in the oscillator application.

We also study the exact resolvent iteration for the full saddle
operator. For any fixed step \(s>0\), both variables are updated
implicitly and jointly. This scheme differs from forward--backward
and semi-implicit iterations. Well-posedness of \eqref{flow}
and nonexpansiveness of the resolvent follow from classical
maximal-monotone theory \cite{Brezis,Rockafellar}.
These properties are the starting point for our estimates.

Knowing that a coefficient is zero at an optimizer does not
give the time when a trajectory reaches zero. A qualitative
finite-identification result may also leave this time unknown.
We seek bounds that can be evaluated from the problem data
\(B,b,\lambda,\gamma\), an observed state, and a chosen index
set \(I\). Our conditions ensure that \(x_i(t)=0\), or
\(x_i^k=0\), for every \(i\in I\) beyond an explicit time
or iteration bound. They apply to the original dynamics,
without clamping selected coordinates.

We distinguish the time when a certificate becomes available
from the deadline it gives. A certificate can predict a future
zero or establish permanence after a zero has been observed.
The numerical experiments report both times.

Safe screening identifies coefficients that vanish at an
optimizer. Ndiaye, Fercoq, Gramfort, and Salmon construct
a duality-gap sphere containing the dual optimum
\cite[Theorem~6]{Ndiaye}. They also prove finite identification
of the equicorrelation set by converging safe active sets
\cite[Proposition~14]{Ndiaye}. Dynamic screening updates
this information during optimization \cite{Bonnefoy}.
Dantas, Soubies, and F\'evotte extend the Gap Safe framework
using local dual strong concavity and refine the safe spheres
\cite[Theorem~5 and Corollary~6]{Dantas}.
Gap radii and optimizer exclusion rules are thus established
tools. We combine them with a bound on the future saddle
trajectory to obtain identification deadlines.

Finite identification has also been studied for optimization
algorithms and dynamical systems. Under partial smoothness,
nondegeneracy, and their algorithmic assumptions, Liang,
Fadili, and Peyr\'e prove finite manifold identification
for primal--dual splitting schemes
\cite[Theorem~3.3]{Liang}. They obtain local linear estimates,
including convergence of the full iterates in the locally
polyhedral case \cite[Theorem~3.11]{Liang}.
As noted in \cite[Remark~3.4(iii)]{Liang}, finite identification
need not provide a bound on the identification index.
Related results are available for thresholding-gradient
methods \cite{Garrigos}. Lewis and Tian study identifiability
of subgradient curves from a geometric perspective
\cite{LewisTian}.

D\'iaz, Izquierdo Lehmann, Lu, and Yang obtain nonasymptotic
identification bounds without strict complementarity,
under their geometric and algorithmic assumptions.
Their functional-inequality framework includes the proximal
point method \cite[Theorems~4.2 and~4.3]{Diaz}.
Their results concern constraint and multiplier structure;
we consider zero primal coefficients in \eqref{flow}.
An epigraph reformulation of the \(\ell^1\) penalty changes
the state space and generally changes the iterated operator.
Applying those results here therefore requires checking
the assumptions and the resulting dynamics.

For the quadratic--\(\ell^1\) model, the saddle operator is
piecewise polyhedral. On a bounded set containing the iterates,
a suitable error bound yields linear convergence of the
proximal point method through metric subregularity.
Lu and Qu discuss polyhedral saddle operators and LASSO
problems and derive the corresponding contraction estimate
\cite[Section~2.2, Lemma~2 and Theorem~2]{LuQu}.
This convergence mechanism is established.
We use restricted-rank estimates to express identification
deadlines directly in the problem constants, without an
unspecified error-bound modulus. The resulting bounds
are computable, although they can be conservative.

For the continuous flow, the squared distance to an equilibrium
is nonincreasing, but its derivative may vanish away from
equilibrium. We add a cross term and the \(\ell^1\) Bregman
remainder. The nonsmooth chain rule lets the remainder cancel
the subgradient coupling introduced by the cross term
(Theorem~\ref{thm:composite}).
This construction is related to the author's earlier work
\cite{SaoudComp,SaoudGeo,SaoudLoc}.
It gives convergence without rank assumptions or primal
uniqueness, and finite identification under strict dual margins.
The same cancellation gives composite dissipation for the
exact full resolvent at every \(s>0\)
(Proposition~\ref{prop:discrete-composite}).
Under strict margins and restricted injectivity, these
estimates also give continuous and discrete decay rates.

Invariant balls bound all future dual states.
For centers \((0,z)\) with
\(\norm{B^\top z}_\infty<\lambda\),
Theorem~\ref{lem:radial} determines the smallest invariant
radius \(R_\sharp(z)\) for the continuous flow within this
family. Balls of radius \(R\ge R_\sharp(z)\) are also invariant
under each exact full resolvent.
If such a ball contains the observed state and
\(|B_i^\top z|+\norm{B_i}R<\lambda\) for \(i\in I\),
Theorem~\ref{thm:ball} gives explicit hitting and permanence
deadlines. This test requires no optimal pair, duality gap,
or restricted singular value.

For a prescribed positive lower bound on
\(\lambda-\norm{B^\top z}_\infty\), center selection is convex
and has an exact finite formulation using semidefinite and
second-order cone constraints
(Propositions~\ref{prop:convex-center-design}
and~\ref{prop:finite-center-sdp}).
Corollary~\ref{cor:optimal-deadline-design} gives a corresponding
formulation for the deadline.
A positive certificate need not exist in the chosen
feasibility class.
Proposition~\ref{prop:obstruction} rules out the zero-center
test for equal-norm columns when a nonzero minimizer exists.
Proposition~\ref{prop:all-center-obstruction} bounds the margins
for every center in this ball family.
The normalized correlated example in
Proposition~\ref{prop:centerexample} attains the largest
coordinate margin within the family and improves the certified
deadline for the same initial state and trajectory.

The gap-based approach bounds future dual states using an
optimizer screen and a distance bound to an equilibrium
(Theorem~\ref{thm:gapcert}).
Proposition~\ref{prop:free} gives a rank-free alternative.
The sharper enclosure and the eventual deadline in
Theorem~\ref{thm:eventual} require \(B_{I^c}\) to be injective
when \(I^c\ne\varnothing\).
The eventual bound applies even when the immediate
future-trajectory margin is nonpositive.
For fixed data, observed state, and certificate parameters,
the discrete bound multiplied by \(s\) converges to the
continuous bound as \(s\downarrow0\)
(Remark~\ref{rem:small}).

The discrete future-trajectory and invariant-ball tests also
allow inexact dual solves, provided primal soft thresholding
is exact and a deterministic bound on the sum of all future
dual residual norms is known.
Small residuals over a finite run do not establish this budget.
The experiments compare certificate availability and deadlines
with observed zeros in a normalized analytic example,
an underdetermined design, a measured dataset, and a simulated
noisy weak-form oscillator regression.
The oscillator illustrates permanence after an observed zero
and the difference between fitted and physical support.

The paper is organized as follows.
Section~\ref{sec:preliminaries} gives the preliminaries.
Section~\ref{sec:lyapunov} develops the Lyapunov and convergence
analysis.
Section~\ref{sec:gap-bounds} presents the gap-based certificates
and resolvent estimates.
Section~\ref{sec:invariant-balls} studies invariant balls,
center selection, and their limitations.
Section~\ref{sec:numerics} gives the numerical method,
experiments, and reproduction protocol.
Section~\ref{sec:conclusions} summarizes the results and scope.
Arguments using bounded subsequences rely on finite-dimensional
compactness; no infinite-dimensional extension is asserted.

\section{Problem setting and preliminaries}\label{sec:preliminaries}
We collect the notation, optimality conditions, and basic estimates
used in the convergence proofs.

We work in the finite-dimensional space
\(\mathcal H=\R^p\times\R^m\), with \(Z=(x,y)\) and
\(\norm Z^2=\norm x^2+\norm y^2\).
Unsubscripted vector norms are Euclidean, and matrix norms are
the corresponding induced norms.
We denote column \(i\) of \(B\) by \(B_i\), set \(L=\norm B\),
and write \(I^c=\{1,\ldots,p\}\setminus I\).
We also use \(\norm w_\infty=\max_i|w_i|\),
\(\norm w_1=\sum_i|w_i|\), and \((a)_+=\max\{a,0\}\).

For \(\lambda,\gamma>0\), recall from \eqref{primal} that
\(g(x)=\lambda\norm{x}_1\) and
\(P(x)=g(x)+\norm{Bx-b}^2/(2\gamma)\).
The Fenchel dual is
\begin{equation}\label{dual}
\mathcal D=\{z\in\R^m:\norm{B^\top z}_\infty\le\lambda\},
\qquad
D(z)=-\ip b z-\frac{\gamma}{2}\norm z^2.
\end{equation}
The bound \(P(x)\ge\lambda\norm{x}_1\) makes \(P\) coercive.
Since \(P\) is continuous, it attains its minimum on \(\R^p\).
At any primal minimizer \(x_*\), the optimality condition gives
\begin{equation}\label{kkt}
y_*=\frac{Bx_*-b}{\gamma},\qquad
\xi_*:=-B^\top y_*\in\partial g(x_*).
\end{equation}
The bound \(\norm{\xi_*}_\infty\le\lambda\) gives
\(y_*\in\mathcal D\).
For every \(x\in\R^p\) and \(z\in\mathcal D\), we have
\[
P(x)\ge
g(x)+\ip{z}{Bx-b}-\frac{\gamma}{2}\norm z^2
\ge D(z).
\]
The first inequality follows by completing the square.
The second follows from dual feasibility.
Both are equalities at \((x_*,y_*)\), since
\(Bx_*-b=\gamma y_*\) and
\(\ip{\xi_*}{x_*}=g(x_*)\), the homogeneity identity for \(g\).
Thus \(\min P=\max_{\mathcal D}D\).
The \(\gamma\)-strong concavity of \(D\) on the convex set
\(\mathcal D\) makes its maximizer \(y_*\) unique.
We write the nonempty equilibrium set as
\[
\Omega=\{(x_*,y_*):x_*\in\operatorname{argmin}P\}.
\]

\begin{definition}[Bregman remainder]\label{def:breg}
Let \(f:\R^n\to(-\infty,+\infty]\) be proper and convex,
\(x,x_*\in\operatorname{dom}f\), and
\(\xi_*\in\partial f(x_*)\).
Its subgradient Bregman remainder is
\[
D_f^{\xi_*}(x,x_*):=
f(x)-f(x_*)-\ip{\xi_*}{x-x_*}\ge0.
\]
This extends the classical Bregman construction to a selected
subgradient \cite{Bregman,Burger}.
It is generally asymmetric and need not be a metric.
For \(f=g\), we write \(h(x):=D_g^{\xi_*}(x,x_*)\).
\end{definition}

By positive homogeneity, the subgradient inequalities at \(0\)
and \(2x_*\) give \(\ip{\xi_*}{x_*}=g(x_*)\).
Hence \(h(x)=g(x)-\ip{\xi_*}{x}\) is convex.
Since \(\xi_*=-B^\top y_*\) is unique, \(h\) is independent
of the chosen primal minimizer.
For any \(x,x^+\in\R^p\) and \(\xi^+\in\partial g(x^+)\),
the subgradient inequality at \(x^+\) gives
\begin{equation}\label{eq:bregman-endpoint}
h(x^+)-h(x)\le
\ip{\xi^+-\xi_*}{x^+-x}.
\end{equation}

\begin{definition}[Fej\'er monotonicity]\label{def:fejer}
A sequence \(Z^k\) is Fej\'er monotone relative to a nonempty
set \(\Omega\) if
\(\norm{Z^{k+1}-Z_*}\le\norm{Z^k-Z_*}\)
for every \(Z_*\in\Omega\) and every \(k\).
For a flow, this means that
\(t\mapsto\norm{Z(t)-Z_*}\) is nonincreasing
for each \(Z_*\in\Omega\).
We use squared distances, which give an equivalent formulation.
Since \(\Omega\ne\varnothing\), Fej\'er monotonicity implies
boundedness.
\end{definition}

The next lemma gives the basic energy estimate and the chain
rule for \(h\) along the flow.

\begin{lemma}[Evolution and nonsmooth chain rule]\label{lem:basic}
For every \((x^0,y^0)\in\mathcal H\), \eqref{flow} has a unique
global strong solution.
It is globally Lipschitz on \([0,\infty)\) and satisfies the
inclusion almost everywhere.
Fix an equilibrium \((x_*,y_*)\), and set
\(u=x-x_*\), \(v=y-y_*\), and
\[
\xi(t)=-\dot x(t)-B^\top y(t)\in\partial g(x(t)),\qquad
q(t)=\xi(t)-\xi_*\quad\text{a.e.}
\]
Then \(\norm q\le2\lambda\sqrt p\) almost everywhere, and
\begin{equation}\label{basicE}
E=\frac12(\norm u^2+\norm v^2),\qquad
\dot E=-\ip u q-\gamma\norm v^2\le0\quad\text{a.e.}
\end{equation}
Moreover, \(h(x(t))\) is locally absolutely continuous, with
\begin{equation}\label{chain}
\begin{aligned}
\frac{d}{dt}h(x(t))
  &=\ip q{\dot x}\quad\text{a.e.},\\
0\le h(x)
  &\le\ip{x-x_*}{\xi-\xi_*}
    \quad(x\in\R^p,\ \xi\in\partial g(x)).
\end{aligned}
\end{equation}
In particular, the solution is bounded and Fej\'er monotone
relative to \(\Omega\).
\end{lemma}

\begin{proof}
\emph{Step 1: well-posedness.}
The operator
\[
A(x,y)=(\partial g(x)+B^\top y,\,-Bx+\gamma y+b)
\]
is the sum of the maximal-monotone operator
\(\partial g\times\{0\}\) and the continuous affine map
\(T(x,y)=(B^\top y,-Bx+\gamma y+b)\).
For \(Z=(x,y)\) and \(\widehat Z=(\widehat x,\widehat y)\),
\[
\ip{T(Z)-T(\widehat Z)}{Z-\widehat Z}
=\gamma\norm{y-\widehat y}^2\ge0.
\]
Thus \(T\) is monotone and everywhere defined.
The sum theorem \cite[Corollary~2.7]{Brezis} makes \(A\)
maximal monotone, with \(\operatorname{dom}A=\mathcal H\).
The Cauchy theorem \cite[Theorem~3.1]{Brezis} then gives
the unique global strong solution and its global Lipschitz
regularity.

\emph{Step 2: energy and remainder bounds.}
The KKT relations give
\(\dot u=-q-B^\top v\) and \(\dot v=Bu-\gamma v\).
Taking scalar products yields \eqref{basicE}, and monotonicity
of \(\partial g\) gives its sign.
Each component of \(\xi\) and \(\xi_*\) lies in
\([-\lambda,\lambda]\), giving the bound on \(q\).
For any \(x\in\R^p\) and \(\xi\in\partial g(x)\), the inequalities
\[
g(x)\ge g(x_*)+\ip{\xi_*}{x-x_*},
\qquad
g(x_*)\ge g(x)+\ip{\xi}{x_*-x}
\]
give the pointwise bound in \eqref{chain}.
Integrating \eqref{basicE} proves Fej\'er monotonicity and
boundedness.

\emph{Step 3: nonsmooth chain rule.}
The function \(g\) is globally Lipschitz.
Thus \(g(x(t))\), and hence \(h(x(t))\), is locally
absolutely continuous.
At almost every time with \(x_i(t)\ne0\), we have
\(\xi_i(t)=\lambda\operatorname{sign}(x_i(t))\), so
\[
\frac{d}{dt}\lambda|x_i(t)|=\xi_i(t)\dot x_i(t).
\]
An absolutely continuous scalar function has zero derivative
almost everywhere on its zero set.
The same identity therefore holds almost everywhere on
\(\{t:x_i(t)=0\}\), for every admissible \(\xi_i(t)\).
Intersecting the finitely many full-measure sets and summing
over \(i\) gives the chain rule for \(g(x(t))\).
Subtracting the derivative of \(\ip{\xi_*}{x-x_*}\) gives
\eqref{chain}, without assuming differentiability of \(g\)
at coordinate zeros.
\end{proof}
\section{Lyapunov analysis and convergence}\label{sec:lyapunov}
The basic distance estimate gives boundedness, but its derivative
can vanish away from equilibrium. Adding a cross term and the
Bregman remainder gives convergence without rank assumptions
and finite identification under strict dual margins.
With restricted injectivity, we also obtain a semiglobal
exponential rate.

\begin{theorem}[Nonsmooth composite dissipation]\label{thm:composite}
Fix an equilibrium \(Z_*=(x_*,y_*)\) and use \(u,v,q,h,E\)
from Lemma~\ref{lem:basic}.
Let \(L=\norm B\) and choose \(\delta>0\) such that
\begin{equation}\label{deltacond}
\delta L<1,\qquad
\delta\left(L^2+\frac{\gamma^2}{2}\right)\le\frac\gamma2.
\end{equation}
Put \(C=\ip{Bu}{v}\) and
\(W_\delta=E+\delta(h-C)\). Then
\begin{equation}\label{compareW}
(1-\delta L)E+\delta h\le W_\delta
\le(1+\delta L)E+\delta h.
\end{equation}
For \(Z=(x,y)\), this function also has the representation
\begin{equation}\label{eq:composite-metric}
M_\delta=
\begin{pmatrix}
I_p&-\delta B^\top\\
-\delta B&I_m
\end{pmatrix},\qquad
W_\delta(Z)=
\frac12\ip{M_\delta(Z-Z_*)}{Z-Z_*}+\delta h(x).
\end{equation}
The matrix satisfies
\((1-\delta L)I_{p+m}\preceq M_\delta
\preceq(1+\delta L)I_{p+m}\).
Hence \(W_\delta\) is \((1-\delta L)\)-strongly convex, with
\begin{equation}\label{eq:composite-subgradient}
\partial W_\delta(Z)=
\bigl\{M_\delta(Z-Z_*)+\delta(\xi-\xi_*,0):
\xi\in\partial g(x)\bigr\}.
\end{equation}
Along the solution, almost everywhere,
\begin{equation}\label{dissW}
\dot W_\delta\le-h-\frac\delta2\norm{Bu}^2
-\frac\gamma2\norm v^2-\delta\norm q^2.
\end{equation}
These inequalities also hold when \(B\) is rank deficient or zero.
\end{theorem}

\begin{proof}
\emph{Step 1: cancellation and dissipation.}
Lemma~\ref{lem:basic} gives
\[
\dot h=-\norm q^2-\ip{Bq}v,\qquad
\dot C=-\ip{Bq}v-\norm{B^\top v}^2
+\norm{Bu}^2-\gamma\ip{Bu}v.
\]
Both derivatives use the same velocity-determined selection \(q\).
Subtracting cancels the \(\ip{Bq}v\) terms and gives
\[
\dot W_\delta=-\ip u q-\gamma\norm v^2+
\delta\{-\norm q^2+\norm{B^\top v}^2
-\norm{Bu}^2+\gamma\ip{Bu}v\}.
\]
Use \(\ip u q\ge h\), \(\norm{B^\top v}\le L\norm v\), and
\(\gamma\ip{Bu}v\le\norm{Bu}^2/2+\gamma^2\norm v^2/2\).
Together with \eqref{deltacond}, these bounds give \eqref{dissW}.

\emph{Step 2: comparison and convexity.}
The bound \(|C|\le L\norm u\norm v\le LE\) gives
\eqref{compareW}.
Expanding the quadratic form gives \eqref{eq:composite-metric}.
For any \(w=(a,c)\),
\[
(1-\delta L)\norm w^2
\le\norm a^2+\norm c^2-2\delta\ip{Ba}c
=\ip{M_\delta w}w
\le(1+\delta L)\norm w^2.
\]
Since \(h\) is convex, the lower bound proves strong convexity
of \(W_\delta\).
The identity \(\partial h(x)=\partial g(x)-\xi_*\) and the
convex sum rule give \eqref{eq:composite-subgradient}.
\end{proof}

\begin{remark}[Scope of the nonsmooth argument]
\label{rem:nonsmooth-scope}
The chain rule in Lemma~\ref{lem:basic} is classical for the
present globally Lipschitz convex penalty.
The homogeneity identity also holds more generally:
for any proper convex positively one-homogeneous function \(f\)
and \(\xi\in\partial f(x)\), the subgradient inequalities at
\(x/2\) and \(2x\) give \(\ip{x}{\xi}=f(x)\).
The coordinatewise absorption estimate and the invariant-ball
calculation below use the separable form
\(g=\lambda\norm\cdot_1\).
The function \(W_\delta\) is locally Lipschitz and nonsmooth
at coordinate zeros.
Its convex representation will be used in the exact resolvent
analysis.
All terms on the right of \eqref{dissW} vanish at every
equilibrium, consistent with possible primal nonuniqueness.
\end{remark}

The next example shows the limits of the simpler candidates
\(E\) and \(E-\delta C\).

\begin{proposition}[Why the Bregman correction matters]
\label{prop:smoothfail}
Let \(m=p=2\), \(B=I_2\), \(\lambda=\gamma=1\), and \(b=(2,0)\).
The pair \((x_*,y_*)=((1,0),(-1,0))\) is an equilibrium.
There are arbitrarily close nonstationary initial states with
\(\dot E(0+)=0\).
For every fixed \(\delta>0\), the smooth candidate \(E-\delta C\)
has a strictly positive derivative at arbitrarily close states
where the vector field is differentiable.
\end{proposition}

\begin{proof}
\emph{Step 1: a vanishing energy derivative.}
For initial offsets \(u=(\eps,0)\), \(v=0\), \(\eps>0\),
the second coordinate pair \((x_2,y_2)\) remains \((0,0)\)
by uniqueness.
Since \(x_1(0)>0\), we have
\(\xi(0+)=(1,0)\) and \(q(0+)=0\).
Thus \(\dot E(0+)=0\), whereas \(\dot v(0+)=u\ne0\).

\emph{Step 2: failure of the smooth cross-term correction.}
For \(u=(\eps,\eps)\), \(v=(0,\sqrt\eps)\), \(\eps>0\),
we have \(\xi=(1,1)\), \(q=(0,1)\), and
\[
\dot E=-2\eps,\qquad
\dot C=-\sqrt\eps-\eps+2\eps^2-\eps^{3/2}.
\]
Hence
\[
\frac d{dt}(E-\delta C)=
-2\eps+\delta(\sqrt\eps+\eps+\eps^{3/2}-2\eps^2)>0
\]
for sufficiently small \(\eps\).
This comparison concerns only these particular Lyapunov candidates.
\end{proof}

The composite estimate controls the quantities needed to find
an equilibrium cluster point.
Fej\'er monotonicity then gives convergence of the whole trajectory.

\begin{theorem}[Rank-free convergence and finite hitting]
\label{thm:convergence}
Every solution of \eqref{flow} converges in \(\R^{p+m}\)
to an equilibrium \((x_\infty,y_*)\).
No injectivity of \(B\) or uniqueness of the primal minimizer
is required.
If \(I_*=\{i:|B_i^\top y_*|<\lambda\}\), then for each \(i\in I_*\)
there exists \(T_i<\infty\) such that
\(x_i(t)=0\) for all \(t\ge T_i\).
\end{theorem}

\begin{proof}
Fix an equilibrium and choose \(\delta>0\) satisfying
\eqref{deltacond}.

\emph{Step 1: integrability.}
Equation \eqref{basicE} bounds \((u,v)\).
Since \(W_\delta\ge0\), integrating \eqref{dissW} and using
monotone convergence as \(T\to\infty\) gives
\[
\int_0^\infty\!
\left(h(x(t))+\tfrac{\delta}{2}\norm{Bu(t)}^2+
\tfrac{\gamma}{2}\norm{v(t)}^2+\delta\norm{q(t)}^2\right)dt
\le W_\delta(0).
\]

\emph{Step 2: pointwise limits.}
The derivatives \(\dot x=-q-B^\top v\) and
\(\dot v=Bu-\gamma v\) are bounded almost everywhere.
Since \(h\) is Lipschitz and squared norms are Lipschitz
on bounded sets, the functions
\(h(x(t))\), \(\norm{Bu(t)}^2\), and \(\norm{v(t)}^2\)
are uniformly continuous.
They are also nonnegative and integrable by Step~1.
A uniformly continuous nonnegative integrable function
tends to zero: otherwise, disjoint intervals would have
a common positive lower bound for their integrals.
Therefore \(h(x(t))\to0\), \(Bu(t)\to0\), and \(v(t)\to0\).

\emph{Step 3: an equilibrium cluster point.}
Boundedness in finite dimensions gives \(t_n\to\infty\)
such that \(x(t_n)\to\bar x\).
Continuity gives \(h(\bar x)=0\), \(B\bar x=Bx_*\), and
\(y(t_n)\to y_*\).
Since \(h(\bar x)=0\), the subgradient inequality
\(g(z)\ge g(x_*)+\ip{\xi_*}{z-x_*}\) becomes
\(g(z)\ge g(\bar x)+\ip{\xi_*}{z-\bar x}\)
for every \(z\in\R^p\).
Hence \(\xi_*\in\partial g(\bar x)\).
Together with \(B\bar x=b+\gamma y_*\), this proves that
\((\bar x,y_*)\) is an equilibrium.

\emph{Step 4: convergence.}
Apply \eqref{basicE} with \((\bar x,y_*)\) as reference.
The distance is nonincreasing and tends to zero along \(t_n\),
so it tends to zero along the whole trajectory.

\emph{Step 5: hitting and permanence.}
For \(i\in I_*\), set \(\eta_i=\lambda-|B_i^\top y_*|>0\).
Since \(y(t)\to y_*\), there is \(\tau_i<\infty\) such that
\(|B_i^\top y(t)|\le\lambda-\eta_i/2\) for all \(t\ge\tau_i\).
Whenever \(x_i(t)\ne0\),
\[
\frac d{dt}|x_i(t)|
=-\lambda-\operatorname{sign}(x_i(t))B_i^\top y(t)
\le-\eta_i/2\quad\text{a.e.}
\]
A connected interval on which \(x_i\ne0\) cannot start
from a zero at or after \(\tau_i\): integration from its
left endpoint would make \(|x_i|\) negative.
Integrating on the interval containing \(\tau_i\), when
\(x_i(\tau_i)\ne0\), therefore gives
\[
|x_i(t)|\le
\left(|x_i(\tau_i)|-\frac{\eta_i}{2}(t-\tau_i)\right)_+
\qquad(t\ge\tau_i).
\]
Thus \(T_i=\tau_i+2|x_i(\tau_i)|/\eta_i\) is valid,
and the coordinate cannot reactivate.
\end{proof}

\begin{remark}\label{rem:strict-margin}
The theorem gives no observable deadline, since \(y_*\)
and the entry times \(\tau_i\) are unknown.
Without strictness, finite hitting may fail.
For \(m=p=1\), \(B=\lambda=1\), and \(\gamma=b=5/2\),
the solution
\(x(t)=a e^{-t/2}\), \(y(t)=-1+(a/2)e^{-t/2}\), \(a>0\),
approaches the saturated equilibrium \((0,-1)\)
without reaching \(x(t)=0\).
\end{remark}

An exponential rate follows when strict margins control \(x_I\)
and injectivity of \(B_{I^c}\) controls the remaining error.

\begin{theorem}[Semiglobal Lyapunov rate]\label{thm:semiglobal}
Fix an equilibrium and a nonempty set
\(I\subseteq\{1,\ldots,p\}\) such that
\(\eta=\lambda-\norm{(B^\top y_*)_I}_\infty>0\).
Put \(S=I^c\).
Then \(x_{*,I}=0\), but \(S\) need not equal \(\supp(x_*)\).
If \(S\ne\varnothing\), assume that \(B_S\) has full column rank
with smallest singular value \(\sigma>0\).
Set \(R=\sqrt{2E(0)}\) and choose \(\delta\) satisfying
\eqref{deltacond}.
For \(S\ne\varnothing\), define
\[
K_R=\frac R\eta\left(1+\frac{2L^2}{\sigma^2}\right),
\quad
C_R=\max\left\{
\frac{1+\delta L}{\sigma^2},\frac{1+\delta L}{2},
\frac{(1+\delta L)K_R}{2}+\delta\right\}.
\]
For \(S=\varnothing\), set \(K_R=R/\eta\) and omit
the first term in \(C_R\).
Then, for \(R>0\),
\[
W_\delta(t)\le W_\delta(0)e^{-\rho_R t},
\qquad
\rho_R=\frac{\min\{1,\delta/2,\gamma/2\}}{C_R}>0.
\]
Consequently,
\(\norm{Z(t)-Z_*}\le
\sqrt{2W_\delta(0)/(1-\delta L)}\,e^{-\rho_Rt/2}\).
The equilibrium is unique under these assumptions.
For \(R=0\), the flow is stationary.
\end{theorem}

\begin{proof}
\emph{Step 1: strict margins and uniqueness.}
By \eqref{kkt}, every primal minimizer vanishes on \(I\)
and has the same image under \(B\).
When \(S\ne\varnothing\), injectivity of \(B_S\) forces
the remaining coordinates to agree.
When \(S=\varnothing\), every primal minimizer is zero.
Since \(y_*\) is unique, the equilibrium is unique.

\emph{Step 2: coercivity on the bounded trajectory.}
Fej\'er monotonicity gives \(\norm u\le R\).
For \(i\in I\), the \(i\)-th Bregman term satisfies
\(\lambda|x_i|-\xi_{*,i}x_i\ge\eta|x_i|\).
The other coordinate remainders are nonnegative, so
\(h(x)\ge\eta\norm{x_I}_1\).
Hence
\(\norm{x_I}^2\le\norm{x_I}_\infty\norm{x_I}_1
\le(R/\eta)h(x)\).
When \(S\ne\varnothing\), we have \(B_Su_S=Bu-B_Ix_I\).
Using \(\norm{B_I}\le L\) and the squared triangle inequality gives
\[
\begin{aligned}
\norm{u_S}^2
&\le\frac2{\sigma^2}
\left(\norm{Bu}^2+L^2\norm{x_I}^2\right),\\
\norm u^2
&\le\frac2{\sigma^2}\norm{Bu}^2+K_R h(x).
\end{aligned}
\]
For \(S=\varnothing\), the same estimate holds without
the first term.

\emph{Step 3: exponential decay.}
Equation \eqref{compareW} gives
\(W_\delta\le C_R(h+\norm{Bu}^2+\norm v^2)\).
Combining this with \eqref{dissW} yields
\(\dot W_\delta\le-\rho_R W_\delta\) almost everywhere.
The a.e. Gronwall inequality gives the Lyapunov estimate.
The lower bound in \eqref{compareW} gives the state estimate.
\end{proof}

\begin{remark}[Strength of restricted injectivity]
\label{rem:restricted-injectivity}
Restricted injectivity applies to all columns in \(S=I^c\),
including inactive coordinates outside \(I\).
It requires \(|S|\le m\), hence \(|I|\ge p-m\).
For a \(40\times60\) design, at least twenty indices must
therefore belong to \(I\), and the remaining columns must
be independent.
A strict margin for one coordinate alone does not imply
this rank condition.

Uniform exponential bounds can fail when \(B\) is rank deficient.
Take \(B=0\), \(b=0\), \(y^0=0\), and \(x^0=a e_1\),
\(a>0\), where \(e_1\) is the first coordinate vector.
Then \(x(t)=(a-\lambda t)_+e_1\) and \(y(t)=0\).
At \(t=a/(2\lambda)\), the ratio
\(W_\delta(t)/W_\delta(0)\) tends to \(1/4\) as \(a\to\infty\).
Thus a positive exponential rate and a multiplicative constant
cannot both be independent of the initial state.
\end{remark}
\section{Gap-based identification bounds}\label{sec:gap-bounds}
The primal--dual gap bounds the distance to the dual optimizer
and can screen coordinates at every primal optimum.
Identification along the dynamics also requires control of
future dual states. We obtain this control from distance bounds
and Fej\'er monotonicity, then derive continuous and discrete
deadlines and establish their small-step limit.

\begin{lemma}[Gap and Bregman enclosure]\label{lem:gap}
Let \(x\in\R^p\), \(z\in\mathcal D\), and
\[
G=P(x)-D(z)\ge0,\qquad r=\sqrt{2G/\gamma}.
\]
For any optimal pair \((x_*,y_*)\),
\begin{equation}\label{gapenclose}
\norm{z-y_*}\le r,\qquad
P(x)-P(x_*)=h(x)+\frac{\norm{B(x-x_*)}^2}{2\gamma}
\le G.
\end{equation}
\end{lemma}

\begin{proof}
\emph{Step 1: dual enclosure.}
Strong concavity of \(D\) and optimality of \(y_*\) on
\(\mathcal D\) give
\[
D(z)\le D(y_*)-\frac\gamma2\norm{z-y_*}^2.
\]
Since \(D(y_*)=P(x_*)\le P(x)\), the first bound follows.

\emph{Step 2: gap decomposition.}
Expanding the squared loss at \(Bx_*-b=\gamma y_*\)
and using \(-B^\top y_*=\xi_*\) gives the identity.
Weak duality gives the final bound.
\end{proof}

The dual enclosure is the classical Gap Safe sphere in this
normalization \cite[Theorem~6]{Ndiaye}.
The next result adds a uniform margin for the future
unmodified trajectory.

\begin{theorem}[Observed-state forward-zero certificate]
\label{thm:gapcert}
At an observed exact state \((x,y)\) at time \(t_0\),
choose \(z\in\mathcal D\) and a nonempty set
\(I\subseteq\{1,\ldots,p\}\).
With \(G,r\) from Lemma~\ref{lem:gap}, put
\[
\eta_i=\lambda-|B_i^\top z|-\norm{B_i}r,\qquad
\eta_I=\min_{i\in I}\eta_i,\quad S=I^c,\quad
M_I=\max_{i\in I}\norm{B_i}.
\]
Assume \(\eta_I>0\).
If \(S\ne\varnothing\), assume that \(B_S\) is injective
with least singular value \(\sigma>0\), and set
\[
a=\left[\norm{x_I}^2+
\frac{\bigl(\sqrt{2\gamma G}+\norm{B_Ix_I}\bigr)^2}
{\sigma^2}\right]^{1/2}.
\]
If \(S=\varnothing\), set \(a=\norm x\).
In both cases, define
\[
d=\sqrt{a^2+(\norm{y-z}+r)^2},\qquad
\eps_i=\lambda-|B_i^\top z|-\norm{B_i}(r+d),\quad
\eps_I=\min_{i\in I}\eps_i.
\]
If \(\eps_I>0\), every exact solution starting from \((x,y)\)
satisfies
\[
x_i(t)=0\quad
\left(i\in I,\quad
t\ge t_0+\max_{j\in I}\frac{|x_j|}{\eps_j}\right).
\]
\end{theorem}

\begin{proof}
\emph{Step 1: screen the optimizers.}
By \eqref{gapenclose},
\[
|B_i^\top y_*|
\le|B_i^\top z|+\norm{B_i}r
\le\lambda-\eta_i.
\]
Thus every primal optimum satisfies \(x_{*,I}=0\).
These coordinates contribute at least
\(\eta_I\norm{x_I}_1\) to \(h(x)\), so
\(\norm{x_I}_1\le G/\eta_I\).
The gap decomposition also gives
\(\norm{B(x-x_*)}\le\sqrt{2\gamma G}\).

\emph{Step 2: enclose an equilibrium.}
For \(S\ne\varnothing\), the identity
\(B_S(x_S-x_{*,S})=B(x-x_*)-B_Ix_I\) gives
\[
\norm{x_S-x_{*,S}}
\le\frac{\sqrt{2\gamma G}+\norm{B_Ix_I}}{\sigma}.
\]
Orthogonality of the coordinate blocks yields
\(\norm{x-x_*}^2=\norm{x_I}^2+
\norm{x_S-x_{*,S}}^2\le a^2\).
For \(S=\varnothing\), every primal optimum is zero,
so \(\norm{x-x_*}=\norm x=a\).
Together with the dual triangle inequality, this gives
\(\norm{(x,y)-(x_*,y_*)}\le d\) for every optimal pair.
Moreover, \(\norm{x_I}\le G/\eta_I\) and
\(\norm{B_Ix_I}\le M_IG/\eta_I\), so
\(a\le(\sqrt{2\gamma G}+M_IG/\eta_I)/\sigma+G/\eta_I\)
when \(S\ne\varnothing\).

\emph{Step 3: a uniform future margin.}
By Fej\'er monotonicity, the distance to the chosen equilibrium
remains at most \(d\).
Since \(\norm{z-y_*}\le r\), every later dual state satisfies
\[
|B_i^\top y(t)|
\le|B_i^\top z|+\norm{B_i}(r+d)
\le\lambda-\eps_i.
\]
The hitting-and-permanence argument in
Theorem~\ref{thm:convergence} gives
\(|x_i(t)|\le(|x_i|-\eps_i(t-t_0))_+\).
Taking the maximum of the individual deadlines proves the result.
\end{proof}

\begin{proposition}[Rank-free alternative]\label{prop:free}
At an observed exact state \((x,y)\) at time \(t_0\),
choose \(z\in\mathcal D\) and a nonempty
\(I\subseteq\{1,\ldots,p\}\).
With \(G,r\) from Lemma~\ref{lem:gap}, set
\[
a_{\rm free}=\norm x+
\frac{\min\{P(x),P(0)\}}{\lambda},\qquad
d_{\rm free}=\sqrt{a_{\rm free}^2+(\norm{y-z}+r)^2}.
\]
Define
\(\eps_{i,\rm free}=\lambda-|B_i^\top z|
-\norm{B_i}(r+d_{\rm free})\).
If \(\eps_{\rm free}:=\min_{i\in I}\eps_{i,\rm free}>0\),
the conclusion of Theorem~\ref{thm:gapcert} holds with
the individual margins \(\eps_{i,\rm free}\).
No rank assumption is required.
\end{proposition}

\begin{proof}
Every optimum satisfies
\(\lambda\norm{x_*}_1\le P(x_*)\le\min\{P(x),P(0)\}\),
where \(P(0)=\norm b^2/(2\gamma)\).
Hence \(\norm{x-x_*}\le a_{\rm free}\).
This replaces the distance estimate in Step~2 of
Theorem~\ref{thm:gapcert}; Step~3 is unchanged.
\end{proof}

The exact full resolvent satisfies analogous distance estimates.
We first prove its Fej\'er and composite decrements, then use
them for identification.

For \(s>0\), the exact resolvent
\(Z^{k+1}=(I+sA)^{-1}Z^k\) is specified by
\begin{equation}\label{resolvent}
\frac{x^{k+1}-x^k}{s}+\xi^{k+1}+B^\top y^{k+1}=0,
\quad\xi^{k+1}\in\partial g(x^{k+1}),\qquad
\frac{y^{k+1}-y^k}{s}=Bx^{k+1}-b-\gamma y^{k+1}.
\end{equation}
It is single-valued and nonexpansive for every \(s>0\).
Write
\(\soft_\tau(w)_i=\operatorname{sign}(w_i)(|w_i|-\tau)_+\).
The primal equation is equivalent to
\begin{equation}\label{soft}
x^{k+1}=\soft_{s\lambda}(x^k-sB^\top y^{k+1}).
\end{equation}
These statements concern the full coupled resolvent,
rather than a forward--backward or semi-implicit splitting.

\begin{lemma}[Discrete Fej\'er decrement]\label{lem:disc}
For each equilibrium \((x_*,y_*)\), set
\(E_k=(\norm{x^k-x_*}^2+\norm{y^k-y_*}^2)/2\).
Then
\begin{equation}\label{discdecrement}
E_{k+1}-E_k\le
-s h(x^{k+1})-s\gamma\norm{y^{k+1}-y_*}^2
-\frac12\norm{x^{k+1}-x^k}^2
-\frac12\norm{y^{k+1}-y^k}^2.
\end{equation}
\end{lemma}

\begin{proof}
Subtract \eqref{kkt} from \eqref{resolvent}.
Pair the primal and dual equations with \(x^{k+1}-x_*\)
and \(y^{k+1}-y_*\), respectively, and add.
The skew products cancel.
Apply
\(2\ip a{a-c}=\norm a^2-\norm c^2+\norm{a-c}^2\)
to both equations and use
\(\ip{x^{k+1}-x_*}{\xi^{k+1}-\xi_*}\ge h(x^{k+1})\).
\end{proof}

Convexity also transfers composite dissipation to implicit Euler.
A related smooth argument appears in
\cite[Theorem~4.1]{ChenWei}.
Here we verify the nonsmooth cancellation.

\begin{proposition}[Composite resolvent dissipation and rate]
\label{prop:discrete-composite}
Fix an equilibrium, \(s>0\), and \(\delta\) satisfying
\eqref{deltacond}.
Put \(u^k=x^k-x_*\), \(v^k=y^k-y_*\),
\(W_k=W_\delta(Z^k)\), and
\(q^{k+1}=\xi^{k+1}-\xi_*\), using the selection in
\eqref{resolvent}.
With \(\Delta Z^k=Z^{k+1}-Z^k\) and
\(\norm w_{M_\delta}^2=\ip{M_\delta w}w\),
\begin{equation}\label{eq:disc-composite}
\begin{aligned}
W_{k+1}-W_k\le{}&
-s\left(h(x^{k+1})+\frac\delta2\norm{Bu^{k+1}}^2
+\frac\gamma2\norm{v^{k+1}}^2
+\delta\norm{q^{k+1}}^2\right)\\
&-\frac12\norm{\Delta Z^k}_{M_\delta}^2.
\end{aligned}
\end{equation}
This inequality requires no rank assumption.
Under the additional hypotheses of Theorem~\ref{thm:semiglobal},
with its constants computed from \(R=\norm{Z^0-Z_*}>0\),
\begin{equation}\label{eq:matched-resolvent-rate}
W_k\le(1+s\rho_R)^{-k}W_0,\qquad
\norm{Z^k-Z_*}\le
\sqrt{\frac{2W_0}{1-\delta L}}\,(1+s\rho_R)^{-k/2}.
\end{equation}
For \(R=0\), the sequence is stationary.
\end{proposition}

\begin{proof}
\emph{Step 1: endpoint convexity.}
Suppress the superscript \(k+1\) on \(u,v,q\), and set
\(p=M_\delta(u,v)+\delta(q,0)\).
The quadratic identity in \eqref{eq:composite-metric} and
the endpoint inequality
\(h(x^{k+1})-h(x^k)\le\ip q{x^{k+1}-x^k}\) give
\[
W_{k+1}-W_k\le\ip p{\Delta Z^k}
-\tfrac12\norm{\Delta Z^k}_{M_\delta}^2.
\]
This uses only convexity of \(h\).

\emph{Step 2: cancellation.}
The coupled update gives
\(\Delta Z^k=s(-q-B^\top v,\,Bu-\gamma v)\).
Hence
\[
\frac1s\ip p{\Delta Z^k}
=-\ip u q-\gamma\norm v^2
+\delta\{-\norm q^2+\norm{B^\top v}^2
-\norm{Bu}^2+\gamma\ip{Bu}v\}.
\]
The same selection \(q\) occurs in both factors.
The estimates from Theorem~\ref{thm:composite} give
\eqref{eq:disc-composite}.

\emph{Step 3: restricted-rank rate.}
Lemma~\ref{lem:disc} gives \(\norm{u^k}\le R\).
The coercivity estimate in Theorem~\ref{thm:semiglobal} yields
\(W_{k+1}\le C_R(h(x^{k+1})+
\norm{Bu^{k+1}}^2+\norm{v^{k+1}}^2)\).
Thus \(W_{k+1}-W_k\le-s\rho_RW_{k+1}\).
Iteration and \eqref{compareW} prove
\eqref{eq:matched-resolvent-rate}.
The argument also permits any prescribed upper bound
\(R\ge\norm{Z^0-Z_*}\) and any certified positive lower
bound on the strict dual margin.
\end{proof}

\begin{theorem}[Exact-resolvent identification]\label{thm:discconv}
For every \(s>0\), the sequence \eqref{resolvent} converges
in \(\R^{p+m}\) to an equilibrium \((x_\infty,y_*)\).
For each \(i\) with \(|B_i^\top y_*|<\lambda\),
the coordinate \(x_i^k\) is zero for all sufficiently large \(k\).
At an observed iterate \(k_0\), assume the hypotheses of
Theorem~\ref{thm:gapcert} hold, including \(\eps_I>0\).
Then
\[
x_i^k=0\quad\left(i\in I,\quad
k\ge k_0+\max_{j\in I}
\left\lceil\frac{|x_j^{k_0}|}{s\eps_j}\right\rceil\right).
\]
Here \(\lceil a\rceil\) is the smallest integer greater
than or equal to \(a\in\R\).
The same conclusion holds with the positive rank-free
margins \(\eps_{i,\rm free}\).
\end{theorem}

\begin{proof}
\emph{Step 1: vanishing observables.}
Summing \eqref{discdecrement} gives
\(h(x^k)\to0\), \(y^k-y_*\to0\), and
\(y^{k+1}-y^k\to0\).
The dual equation then yields
\[
B(x^{k+1}-x_*)=(y^{k+1}-y^k)/s
+\gamma(y^{k+1}-y_*)\longrightarrow0.
\]

\emph{Step 2: convergence.}
Boundedness and finite-dimensional compactness give
a convergent subsequence.
The subgradient argument in Step~3 of
Theorem~\ref{thm:convergence}, together with the dual identity
above, shows that its limit is an equilibrium.
Fej\'er monotonicity relative to that limit gives convergence
of the whole sequence.

\emph{Step 3: hitting and permanence.}
Use the scalar inequality
\[
|\soft_{s\lambda}(a-sc)|
\le(|a|+s|c|-s\lambda)_+.
\]
For a strict index, convergence of \(y^k\) eventually gives
\(|B_i^\top y^{k+1}|\le\lambda-\eta_i/2\), where
\(\eta_i=\lambda-|B_i^\top y_*|>0\).
Each subsequent step either sets the coordinate to zero
or decreases its magnitude by at least \(s\eta_i/2\).
Once zero, it remains zero.

At an observed iterate, the gap enclosure gives
\(|B_i^\top y^{k+1}|\le\lambda-\eps_i\) for all \(k\ge k_0\).
Induction yields
\(|x_i^{k_0+j}|\le(|x_i^{k_0}|-js\eps_i)_+\),
which proves the deadline.
The rank-free enclosure gives the stated variant.
\end{proof}

A known budget for all future dual residuals enlarges the
distance enclosure. Exact primal soft thresholding is still required.

\begin{theorem}[Known summable dual residuals]\label{thm:inexact}
Fix \(s>0\), an initial index \(k_0\), and computed states
satisfying \eqref{soft} exactly at every future step.
Define the dual-equation residual by
\[
\omega^{k+1}=(1+s\gamma)y^{k+1}-y^k
-s(Bx^{k+1}-b).
\]
Assume a known deterministic bound
\(E_{\rm err}\ge\sum_{k=k_0}^\infty\norm{\omega^{k+1}}\).
At state \(k_0\), let \(r,z\) be as in Lemma~\ref{lem:gap}.
Let \(d_{\rm obs}\) be a computable upper bound on the distance
to one fixed optimal pair, such as \(d\) or \(d_{\rm free}\).
For a nonempty \(I\subseteq\{1,\ldots,p\}\), define
\[
\eps_{i,E}=\lambda-|B_i^\top z|-\norm{B_i}
(r+d_{\rm obs}+E_{\rm err}),\qquad
\eps_{I,E}=\min_{i\in I}\eps_{i,E}.
\]
If \(\eps_{I,E}>0\), then \(x_i^k=0\) for \(i\in I\) and all
\[
k\ge k_0+\max_{j\in I}
\left\lceil\frac{|x_j^{k_0}|}{s\eps_{j,E}}\right\rceil.
\]
\end{theorem}

\begin{proof}
\emph{Step 1: residual enclosure.}
Exact \eqref{soft} and the residual identity give
\(Z^{k+1}=J_{sA}(Z^k+(0,\omega^{k+1}))\).
Since \(J_{sA}\) is nonexpansive and fixes the chosen
optimal pair \(Z_*\),
\[
\norm{Z^{k+1}-Z_*}
\le\norm{Z^k-Z_*}+\norm{\omega^{k+1}}.
\]
Induction bounds every future distance by
\(d_{\rm obs}+E_{\rm err}\).

\emph{Step 2: coordinate shrinkage.}
Combine this enclosure with \(\norm{z-y_*}\le r\)
and apply \eqref{soft} as in Theorem~\ref{thm:discconv}.
\end{proof}

\begin{remark}\label{rem:dual-residual-scope}
Small residuals observed during a finite run do not establish
the infinite-future budget \(E_{\rm err}\).
The theorem does not cover approximate primal thresholding
or arbitrary perturbations of \eqref{resolvent}.
The exact composite rate in
Proposition~\ref{prop:discrete-composite} is not asserted
for this inexact sequence.
\end{remark}

The optimizer screen \(\eta_I>0\) may hold even when the
observed future margin \(\eps_I\) is nonpositive.
Under restricted injectivity, composite decay then gives
a computable eventual deadline.

\begin{theorem}[Computable eventual deadlines]\label{thm:eventual}
At an observed exact state \((x^0,y^0)\) at time \(t_0\),
choose a dual-feasible \(z\) and a nonempty
\(I\subseteq\{1,\ldots,p\}\).
Compute \(G,r,\eta_I,S\) as in Theorem~\ref{thm:gapcert}
and assume \(\eta_I>0\).
No positivity of the observed future margin \(\eps_I\) is required.

If \(B=0\), direct shrinkage gives \(x_I(t)=0\) for
\(t\ge t_0+\norm{x_I^0}_\infty/\lambda\).
The exact resolvent identifies \(I\) within
\(\lceil\norm{x_I^0}_\infty/(s\lambda)\rceil\) steps.

For \(B\ne0\), assume that \(B_S\) is injective when
\(S\ne\varnothing\), with least singular value \(\sigma>0\).
Let \(d\) be the distance bound in Theorem~\ref{thm:gapcert}.
If \(d=0\), the observed state is an equilibrium
and \(x_I^0=0\).
For \(d>0\), put \(L=\norm B\),
\(M=\max_{i\in I}\norm{B_i}\), and choose \(\delta\)
satisfying \eqref{deltacond}.
When \(S\ne\varnothing\), define
\[
K_d=\frac d{\eta_I}(1+2L^2/\sigma^2),\qquad
C_W=\max\left\{\frac{1+\delta L}{\sigma^2},
\frac{1+\delta L}{2},
\frac{(1+\delta L)K_d}{2}+\delta\right\}.
\]
When \(S=\varnothing\), set \(K_d=d/\eta_I\)
and omit the first term in \(C_W\).
Define
\[
\rho=\frac{\min\{1,\delta/2,\gamma/2\}}{C_W},
\quad U=\frac{1+\delta L}{2}d^2+\delta G,
\quad V=\sqrt{\frac{2U}{1-\delta L}}.
\]
For \(M>0\), put \(\ell=\max\{0,\log(2MV/\eta_I)\}\);
for \(M=0\), put \(\ell=0\).
Set \(T=2\ell/\rho\).
Then the continuous solution satisfies \(x_I(t)=0\) for all
\begin{equation}\label{contdeadline}
t\ge t_0+T+\frac{2\min\{d,V e^{-\rho T/2}\}}{\eta_I}.
\end{equation}
For the exact resolvent with fixed \(s>0\), counted from
the same observed state, define
\begin{equation}\label{discdeadline}
N_s=\left\lceil\frac{2\ell}{\log(1+s\rho)}\right\rceil,
\qquad
K_s=N_s+\left\lceil
\frac{2\min\{d,V(1+s\rho)^{-N_s/2}\}}{s\eta_I}
\right\rceil.
\end{equation}
Then \(x_I^k=0\) permanently for every \(k\ge K_s\).
\end{theorem}

\begin{proof}
Assume \(B\ne0\) and \(d>0\).

\emph{Step 1: a common decay estimate.}
The construction in Theorem~\ref{thm:gapcert} encloses a fixed
optimal pair within distance \(d\).
Its optimizer screen gives a strict dual margin of at least
\(\eta_I\) on \(I\).
Both Fej\'er inequalities keep the future distance at most \(d\).
The restricted-rank estimate with radius \(d\) and margin
lower bound \(\eta_I\) gives
\(W_\delta\le C_W(h+\norm{Bu}^2+\norm v^2)\).
Theorem~\ref{thm:composite} and
Proposition~\ref{prop:discrete-composite} therefore yield
\[
W_\delta(t_0+\tau)\le e^{-\rho\tau}W_\delta(t_0),
\qquad W_k\le(1+s\rho)^{-k}W_0.
\]
Lemma~\ref{lem:gap} gives \(h(x^0)\le G\).
Hence \eqref{compareW} gives
\(W_\delta(t_0)\le U\) and \(W_0\le U\).
Combining the Fej\'er and exponential bounds yields
\begin{equation}\label{eq:observed-matched-enclosure}
\begin{aligned}
\norm{Z(t_0+\tau)-Z_*}&\le
\min\{d,V e^{-\rho\tau/2}\},\qquad \tau\ge0,\\
\norm{Z^k-Z_*}&\le
\min\{d,V(1+s\rho)^{-k/2}\},\qquad k\ge0.
\end{aligned}
\end{equation}

\emph{Step 2: the continuous deadline.}
By the definition of \(T\), the optimal margin and
\eqref{eq:observed-matched-enclosure} give
\(|B_i^\top y(t)|\le\lambda-\eta_I/2\)
for every \(i\in I\) and \(t\ge t_0+T\).
The hitting inequality then needs at most
\(2\norm{x_I(t_0+T)}_\infty/\eta_I\) additional time.
Since \(x_{*,I}=0\), the coordinate norm is bounded by
\(\min\{d,V e^{-\rho T/2}\}\).
This proves \eqref{contdeadline}.

\emph{Step 3: the discrete deadline.}
For \(M>0\), the definition of \(N_s\) gives
\(MV(1+s\rho)^{-N_s/2}\le\eta_I/2\).
For every \(i\in I\) and \(k\ge N_s\),
\[
|B_i^\top y^{k+1}|
\le\lambda-\eta_I+
MV(1+s\rho)^{-(k+1)/2}
\le\lambda-\eta_I/2.
\]
For \(M=0\), this margin holds immediately.
Equation~\eqref{soft} gives
\(|x_i^{k+1}|\le(|x_i^k|-s\eta_I/2)_+\).
At \(N_s\), the coordinate magnitudes are at most
\(A_s:=\min\{d,V(1+s\rho)^{-N_s/2}\}\).
Induction yields
\(|x_i^{N_s+j}|\le(A_s-js\eta_I/2)_+\)
for every integer \(j\ge0\), proving \eqref{discdeadline}
and permanence.
The cases \(B=0\) and \(d=0\) follow directly from
\eqref{flow} and \eqref{soft}.
\end{proof}

\begin{remark}[Continuous-time limit of the deadline]
\label{rem:small}
Under Theorem~\ref{thm:eventual}, fix the data, observed state,
and parameters \(z,I,\delta\), with \(B\ne0\) and \(d>0\).
As \(s\downarrow0\), the discrete deadline multiplied by \(s\)
converges to the continuous bound:
\begin{equation}\label{eq:deadline-continuous-limit}
sK_s\longrightarrow
T+\frac{2\min\{d,V e^{-\rho T/2}\}}{\eta_I}.
\end{equation}
Indeed, put \(b_s=\log(1+s\rho)\), so \(b_s/s\to\rho\).
For \(\ell>0\), the definition of \(N_s\) gives
\(2\ell\le N_sb_s<2\ell+b_s\) and
\(2\ell s/b_s\le sN_s<2\ell s/b_s+s\).
Hence \(sN_s\to T\), \(N_sb_s\to\rho T\), and
\(A_s\to\min\{d,V e^{-\rho T/2}\}\).
After multiplication by \(s\), the second ceiling changes
\(2A_s/\eta_I\) by less than \(s\), proving
\eqref{eq:deadline-continuous-limit}.
For \(\ell=0\), \(N_s=0\), and the limit follows directly.

For \(B=0\), the discrete shrinkage bound multiplied by \(s\)
tends to \(\norm{x_I^0}_\infty/\lambda\).
For \(d=0\), both deadlines are zero.
These are sufficient upper bounds and may exceed the actual
identification times. Small values of \(\sigma\) or \(\eta_I\)
can make them conservative.
\end{remark}

\begin{remark}[Established polyhedral error-bound comparison]
\label{rem:polyhedral-comparison}
Write \(\Omega=A^{-1}(0)\).
The graph of \(\partial\norm{\cdot}_1\) is a finite union
of polyhedra, so adding the linear saddle terms makes
\(A\) piecewise polyhedral.
Suppose a metric-subregularity modulus \(\kappa_R\) is
available on a fixed bounded set containing the iterates:
\[
\operatorname{dist}(Z,\Omega)\le
\kappa_R\operatorname{dist}(0,A(Z)).
\]
For \(Z^{k+1}=J_{sA}Z^k\),
\((Z^k-Z^{k+1})/s\in A(Z^{k+1})\).
The set \(\Omega=\arg\min P\times\{y_*\}\) is nonempty,
closed, and convex.
Applying the resolvent Fej\'er inequality with the projection
of \(Z^k\) onto \(\Omega\) gives
\[
\operatorname{dist}(Z^{k+1},\Omega)^2
\le\operatorname{dist}(Z^k,\Omega)^2
-\norm{Z^k-Z^{k+1}}^2.
\]
The error bound then yields
\[
\operatorname{dist}(Z^{k+1},\Omega)\le
\frac{\kappa_R}{\sqrt{\kappa_R^2+s^2}}
\operatorname{dist}(Z^k,\Omega).
\]
This linear convergence mechanism is established in
\cite[Section~2.2, Lemma~2 and Theorem~2]{LuQu}.
The modulus \(\kappa_R\) may be difficult to compute.
Theorem~\ref{thm:eventual} instead gives explicit constants
for a restricted-rank certificate and converts decay into
an identification deadline, including its continuous-time limit.
\end{remark}
\section{Invariant balls and the choice of center}\label{sec:invariant-balls}
An invariant ball controls every later dual state.
Using centers \((0,z)\), with \(z\) strictly dual feasible,
gives identification deadlines without rank assumptions
or optimizer knowledge.
We find the smallest invariant radius for the continuous
flow within this family and optimize the center through
an exact finite conic formulation.
An all-center obstruction and an example attaining its bound
describe the limits of the construction.
We retain the finite-dimensional setting of
Section~\ref{sec:preliminaries}, with \(\lambda,\gamma>0\).

Let \(z\in\R^m\) be strictly dual feasible, and put
\[
\mu_z:=\lambda-\norm{B^\top z}_\infty>0,\qquad
\tau_z:=\norm{b+\gamma z}.
\]
Define \(C_z=(0,z)\),
\(H_z(x,y)=(\norm x^2+\norm{y-z}^2)/2\), and
\begin{equation}\label{Rsharp}
R_\sharp(z)=\max_{0\le t\le\tau_z/\gamma}
\left[t^2+
\left(\frac{\tau_z t-\gamma t^2}{\mu_z}\right)^2
\right]^{1/2},
\end{equation}
with \(R_\sharp=0\) when \(\tau_z=0\).
This requires only a one-dimensional maximization of known data.
A simple upper bound is
\(R_\sharp\le[
(\tau_z/\gamma)^2+
(\tau_z^2/(4\gamma\mu_z))^2]^{1/2}\).
Write \(K(z,R)=\{Z:\norm{Z-C_z}\le R\}\), \(R\ge0\).

\begin{theorem}[Exact invariant-ball radius]\label{lem:radial}
For every \((x,y)\) and \(\xi\in\partial g(x)\),
\begin{equation}\label{radial}
\ip{(x,y)-C_z}{(\xi+B^\top y,-Bx+\gamma y+b)}
\ge\mu_z\norm x+\gamma\norm{y-z}^2
-\tau_z\norm{y-z}.
\end{equation}
The lower bound is attained at every prescribed pair of norms
\(\norm x=a\), \(\norm{y-z}=t\).
For every \(R\ge0\), the ball \(K(z,R)\) is forward invariant
for \eqref{flow} from every initial point in the ball
if and only if \(R\ge R_\sharp(z)\).
Equivalently,
\begin{equation}\label{eq:radius-angular}
R_\sharp(z)=\sup_{0<u\le1}
\frac{u\tau_z-\mu_z\sqrt{1-u^2}}{\gamma u^2}.
\end{equation}
Every ball with \(R\ge R_\sharp(z)\) is also invariant
under each exact full resolvent \(J_{sA}\), \(s>0\).
Minimality concerns the continuous flow within this center
family, not a particular fixed resolvent step.
\end{theorem}

\begin{proof}
\emph{Step 1: exact radial minimization.}
Put \(w=y-z\), \(a=\norm x\), and \(t=\norm w\).
Cancellation and \(\ip x\xi=g(x)\) give the continuous,
selection-independent pairing
\[
Q_z(x,w)=\lambda\norm x_1+\ip{B^\top z}x
+\gamma\norm w^2+\ip{b+\gamma z}w.
\]
At fixed \(a,t\), its minimum is
\(\mu_za+\gamma t^2-\tau_zt\).
For the primal term, choose a coordinate vector of length \(a\),
with sign opposite to a component attaining
\(\norm{B^\top z}_\infty\).
If \(B^\top z=0\), any signed coordinate vector suffices.
For the dual term, take \(w=-t(b+\gamma z)/\tau_z\)
when \(\tau_z>0\); otherwise, any direction suffices.
This proves \eqref{radial} and its sharpness.

\emph{Step 2: the boundary radius.}
For \(R>0\), write
\(a=R\sqrt{1-u^2}\), \(t=Ru\), \(0\le u\le1\).
The minimum boundary pairing is
\(R[\mu_z\sqrt{1-u^2}+\gamma Ru^2-\tau_zu]\).
It is nonnegative for every \(u\) exactly when \(R\)
is at least the supremum in \eqref{eq:radius-angular}.

For \(\tau_z>0\), negative pairings satisfy
\(0<t<\tau_z/\gamma\) and
\(0\le a<(\tau_zt-\gamma t^2)/\mu_z\).
The supremum of their attainable radii is therefore
\eqref{Rsharp}, including the endpoint values by limits.
Thus the two radius formulas agree.
For \(\tau_z=0\), both equal zero.
Applying the same argument to nonpositive pairings shows
that \(Q_z>0\) whenever \(\norm{(x,w)}>R_\sharp(z)\).

\emph{Step 3: continuous sufficiency and necessity.}
Along a strong solution, \(H_z\) is locally absolutely continuous
and \(\dot H_z=-Q_z(x,y-z)\) almost everywhere.
For \(R\ge R_\sharp(z)\), the positive-part chain rule gives
\(\frac{d}{dt}(H_z-R^2/2)_+\le0\) almost everywhere.
A solution starting in \(K(z,R)\) therefore remains there.

If \(0<R<R_\sharp(z)\), Step~2 gives a boundary point
with \(Q_z<0\).
Continuity of \(Q_z\) and the solution preserves this sign
for a short positive time.
Integrating \(\dot H_z=-Q_z\) gives immediate exit.
This argument requires no differentiability of the vector field.

For \(R=0\), the singleton is invariant exactly when
\(C_z\) is an equilibrium.
Since \(-B^\top z\in\partial g(0)\), this is equivalent
to \(b+\gamma z=0\), or \(\tau_z=0\).

\emph{Step 4: exact resolvents.}
For \(A^{k+1}=(Z^k-Z^{k+1})/s\in A(Z^{k+1})\),
the quadratic identity gives
\[
H_z(Z^{k+1})-H_z(Z^k)
=-s\ip{Z^{k+1}-C_z}{A^{k+1}}
-\tfrac12\norm{Z^{k+1}-Z^k}^2.
\]
If a step from \(K(z,R)\), \(R\ge R_\sharp(z)\),
landed outside the ball, its pairing would be strictly positive.
The displayed change in \(H_z\) would then be negative,
contradicting the increase in radius.
\end{proof}

\begin{lemma}[Closed-form evaluation of the radius]\label{lem:closed}
Put \(c_\sharp=(11+5\sqrt5)/2\).
Then \(R_\sharp(z)=\tau_z/\gamma\) if and only if
\(\tau_z^2\le c_\sharp\mu_z^2\).
If \(\tau_z^2>c_\sharp\mu_z^2\), set
\[
t_-=
\frac{3\tau_z-\sqrt{\tau_z^2-8\mu_z^2}}{4\gamma}.
\]
Then \(R_\sharp(z)\) is the expression in \eqref{Rsharp}
evaluated at \(t_-\).
\end{lemma}

\begin{proof}
\emph{Step 1: endpoint dominance.}
The case \(\tau_z=0\) is immediate.
Otherwise, put \(v=\gamma t/\tau_z\) and
\(k=\tau_z^2/\mu_z^2\).
The squared maximand divided by \((\tau_z/\gamma)^2\)
becomes \(v^2+kv^2(1-v)^2\).
It is at most one on \([0,1]\) exactly when
\(k\le\inf_{0<v<1}(1+v)/(v^2(1-v))\).
Differentiation gives the unique minimum at
\(v=(\sqrt5-1)/2\), with value \(c_\sharp\).

\emph{Step 2: the interior maximum.}
Above this threshold, the maximum exceeds the endpoint value.
The interior stationary points solve
\(2\gamma^2t^2-3\gamma\tau_zt+\tau_z^2+\mu_z^2=0\).
Both roots lie in \((0,\tau_z/\gamma)\).
The smaller root \(t_-\) is a local maximum;
the larger root is a local minimum.
This proves the formula.
\end{proof}

An invariant ball containing the observed state gives uniform
coordinate margins and hence identification deadlines.

\begin{theorem}[Prospective permanent-zero certificate]\label{thm:ball}
At a state \((x^0,y^0)\) at time \(t_0\), choose a
strictly dual-feasible \(z\) and set
\[
R=\max\big\{\sqrt{\norm{x^0}^2+\norm{y^0-z}^2},
R_\sharp(z)\big\}.
\]
For a nonempty \(I\subseteq\{1,\ldots,p\}\), assume
\begin{equation}\label{ballmargin}
\eps_{i,z}=\lambda-|B_i^\top z|-\norm{B_i}R,
\qquad \eps_{I,z}=\min_{i\in I}\eps_{i,z}>0.
\end{equation}
Then the exact flow satisfies
\[
x_i(t)=0\quad
\left(i\in I,\quad
t\ge t_0+\max_{j\in I}\frac{|x_j^0|}{\eps_{j,z}}\right).
\]
For the exact full resolvent starting at the same state,
\(x_i^k=0\) for every \(i\in I\) and
\(k\ge\max_{j\in I}\lceil|x_j^0|/(s\eps_{j,z})\rceil\),
for every \(s>0\).
All inputs are observable; no optimizer or restricted
singular value is needed.
\end{theorem}

\begin{proof}
\emph{Step 1: uniform margins.}
Theorem~\ref{lem:radial} bounds
\(\norm{y(t)-z}\) and \(\norm{y^k-z}\) by \(R\)
at every later time or index.
Hence
\(|B_i^\top y|\le|B_i^\top z|+\norm{B_i}R
\le\lambda-\eps_{i,z}\).

\emph{Step 2: hitting and permanence.}
Apply the hitting-and-permanence argument from
Theorem~\ref{thm:convergence} to the flow,
or iterate \eqref{soft} for the resolvent.
\end{proof}

\begin{corollary}[Inexact invariant-ball certificate]\label{cor:ballerr}
Under the exact thresholding and known summable dual-residual
budget assumptions of Theorem~\ref{thm:inexact},
suppose the initial state lies in the ball of radius \(R\)
from Theorem~\ref{thm:ball}.
If
\[
\widetilde\eps_{i,z}=\lambda-|B_i^\top z|
-\norm{B_i}(R+E_{\rm err}),\qquad
\min_{i\in I}\widetilde\eps_{i,z}>0,
\]
every coordinate in \(I\) vanishes permanently within
\(\max_{i\in I}\lceil|x_i^0|/(s\widetilde\eps_{i,z})\rceil\)
steps from the observed state.
\end{corollary}

\begin{proof}
\emph{Step 1: an inflated enclosure.}
The exact resolvent maps \(K=K(z,R)\) into itself.
Nonexpansiveness gives
\[
\operatorname{dist}(Z^{k+1},K)\le
\operatorname{dist}(Z^k+(0,\omega^{k+1}),K)
\le\operatorname{dist}(Z^k,K)+\norm{\omega^{k+1}}.
\]
For the first inequality, compare the shifted state with
its projection onto \(K\), whose resolvent image remains in \(K\).
Induction gives \(\norm{Z^k-C_z}\le R+E_{\rm err}\).

\emph{Step 2: coordinate shrinkage.}
Apply \eqref{soft} with the inflated margins.
\end{proof}

We now choose the center and radius jointly.
A fixed positive lower bound on dual feasibility
keeps the admissible set closed.

\begin{proposition}[Convex center selection]
\label{prop:convex-center-design}
Fix \(Z^0=(x^0,y^0)\), a nonempty
\(I\subseteq\{1,\ldots,p\}\), and
\(0<\underline\mu\le\lambda\).
Put
\(\mathcal D_{\underline\mu}
=\{z:\norm{B^\top z}_\infty\le\lambda-\underline\mu\}\).
The problem
\begin{equation}\label{eq:center-design}
\begin{aligned}
\max_{z,R,\eps}\quad&\eps\\
\text{subject to}\quad&
z\in\mathcal D_{\underline\mu},\quad R\ge0,
\quad\norm{(x^0,y^0-z)}\le R,\\
&R\ge R_\sharp(z),\\
&|B_i^\top z|+\norm{B_i}R+\eps\le\lambda
\quad(i\in I)
\end{aligned}
\end{equation}
is convex, feasible, and attains its maximum.
The variable \(\eps\in\R\) is unrestricted in sign.
Its optimal value is positive exactly when a positive
common-margin certificate exists in this feasibility class.
If every selected column is zero, the optimal value is \(\lambda\).
\end{proposition}

\begin{proof}
\emph{Step 1: convex invariance constraints.}
By \eqref{eq:radius-angular}, \(R\ge R_\sharp(z)\)
is equivalent to
\begin{equation}\label{eq:convex-invariance}
u\norm{b+\gamma z}+\sqrt{1-u^2}\norm{B^\top z}_\infty
-\lambda\sqrt{1-u^2}-\gamma Ru^2\le0
\quad(0\le u\le1).
\end{equation}
For \(R=0\), the constraint at \(u=1\) forces
\(b+\gamma z=0\), as required for singleton invariance.
Each inequality defines a closed convex set in \((z,R)\).
The remaining constraints are also convex and closed.

\emph{Step 2: feasibility and certificate meaning.}
Choose \(z=0\), a sufficiently large \(R\), and a sufficiently
negative \(\eps_0\).
This gives a feasible baseline.
Every feasible point with \(\eps>0\) gives a certificate
by Theorem~\ref{thm:ball}.
Conversely, every common-margin ball certificate in this
feasibility class gives a feasible point.

\emph{Step 3: attainment.}
If \(B_j\ne0\) for some \(j\in I\), every feasible point
with \(\eps\ge\eps_0\) satisfies
\(R\le(\lambda-\eps_0)/\norm{B_j}\).
State containment then bounds \(z\), and
\(\eps_0\le\eps\le\lambda\).
This feasible superlevel set is closed and bounded in finite
dimensions, hence compact.
A maximizing sequence therefore has a feasible subsequential
limit attaining the maximum.
If all selected columns vanish, the constraints reduce to
\(\eps\le\lambda\), attained by the baseline ball.
\end{proof}

The continuum of invariance constraints has an exact finite
representation using two positive-semidefinite blocks.

\begin{proposition}[Exact finite conic formulation]
\label{prop:finite-center-sdp}
For \(z\in\mathcal D_{\underline\mu}\) and \(R\ge0\),
the condition \(R\ge R_\sharp(z)\) is equivalent to the
existence of \(\mu,\tau\in\R\), \(Q\in\mathbb S^3_+\),
and \(H\in\mathbb S^2_+\) satisfying
\begin{equation}\label{eq:radius-auxiliary}
\mu\ge\underline\mu,\qquad
\mu+\norm{B^\top z}_\infty\le\lambda,\qquad
\norm{b+\gamma z}\le\tau,
\end{equation}
and, with indices starting at zero,
\begin{equation}\label{eq:quartic-coefficients}
\begin{aligned}
Q_{00}&=\mu,&
2Q_{01}+H_{00}&=-2\tau,\\
2Q_{02}+Q_{11}+2H_{01}-H_{00}&=4\gamma R,&
2Q_{12}+H_{11}-2H_{01}&=-2\tau,\\
Q_{22}-H_{11}&=-\mu.
\end{aligned}
\end{equation}
Here \(\mathbb S^d_+\) is the cone of symmetric
positive-semidefinite \(d\times d\) matrices.
Thus \eqref{eq:center-design} has an exact finite
semidefinite formulation with second-order cone constraints.
\end{proposition}

\begin{proof}
\emph{Step 1: scalar polynomial.}
The substitution \(u=2r/(1+r^2)\) maps \([0,1]\) onto
\([0,1]\), with
\(\sqrt{1-u^2}=(1-r^2)/(1+r^2)\).
Multiplying
\(\mu\sqrt{1-u^2}+\gamma Ru^2-\tau u\ge0\)
by \((1+r^2)^2>0\) gives
\begin{equation}\label{eq:invariance-quartic}
p(r)=\mu(1-r^4)+4\gamma Rr^2-2\tau(r+r^3)\ge0
\quad(0\le r\le1).
\end{equation}

\emph{Step 2: interval positivity.}
The classical representation
\cite[Theorem~3.72]{Parrilo} makes quartic nonnegativity
on \([0,1]\) equivalent to
\[
p(r)=v_2(r)^\top Qv_2(r)
+r(1-r)v_1(r)^\top Hv_1(r),
\quad v_2=(1,r,r^2)^\top,\quad v_1=(1,r)^\top,
\]
with \(Q\succeq0\) and \(H\succeq0\).
Expansion gives \eqref{eq:quartic-coefficients}.

\emph{Step 3: exact projection.}
If \(R\ge R_\sharp(z)\), choose
\(\mu=\mu_z\) and \(\tau=\tau_z\).
The preceding steps give the required representation.
Conversely, \eqref{eq:radius-auxiliary} gives
\(\mu\le\mu_z\) and \(\tau\ge\tau_z\).
Replacing the auxiliary values by the actual values increases
the scalar left-hand side, so \eqref{eq:convex-invariance} holds.
The formulation is therefore exact after projection onto \((z,R)\).
The remaining norm constraints are conic, and the coefficient
equations are affine.
\end{proof}

Maximizing a common margin need not minimize the deadline
when the initial coordinate magnitudes differ.
We therefore optimize the reciprocal deadline directly.

\begin{corollary}[Inverse-deadline center selection]
\label{cor:optimal-deadline-design}
Assume \(\norm{x_I^0}_\infty>0\).
Retain the feasibility, containment, and invariance constraints
in \eqref{eq:center-design}, and maximize \(\beta\ge0\) subject to
\begin{equation}\label{eq:inverse-deadline-design}
|B_i^\top z|+\norm{B_i}R+\beta|x_i^0|\le\lambda
\quad(i\in I).
\end{equation}
This problem is convex and has the same finite conic formulation.
It may be infeasible; when feasible, it attains its maximum.
Every feasible \(\beta>0\) certifies permanent zeros on \(I\)
by time \(t_0+1/\beta\), or within
\(\lceil1/(s\beta)\rceil\) exact resolvent steps.
\end{corollary}

\begin{proof}
\emph{Step 1: deadlines and initially zero coordinates.}
For \(x_i^0\ne0\), the margin is at least
\(\beta|x_i^0|>0\), giving a deadline of at most \(1/\beta\).
For \(x_i^0=0\), the constraint gives
\(|B_i^\top y|\le\lambda\) throughout the enclosing ball.
On any nonzero interval, \(d|x_i|/dt\le0\) almost everywhere,
so such an interval cannot begin at zero.
The soft-threshold inequality likewise preserves a zero iterate.

\emph{Step 2: convexity and attainment.}
The new constraints are convex because \(x_i^0\) is fixed.
If some selected \(B_j\ne0\), feasibility gives
\(R\le\lambda/\norm{B_j}\), and containment bounds \(z\).
An initially nonzero selected coordinate gives
\(\beta\le\lambda/|x_i^0|\).
The feasible set is therefore compact, so the maximum is attained.
If all selected columns vanish, the attained maximum is
\(\lambda/\norm{x_I^0}_\infty\).
\end{proof}

\begin{remark}[Scope and verification of center selection]
\label{rem:center-synthesis-scope}
Optimality in \eqref{eq:center-design} is relative to
the prescribed lower bound \(\underline\mu\).
A nonpositive optimum does not exclude successful centers
closer to the dual boundary.
If \(x_I^0=0\), check permanence using nonnegative margins
directly; the inverse-deadline objective is unbounded whenever
its constraints are feasible.

For a fixed known dual-residual budget \(E_{\rm err}\),
replace \(R\) by \(R+E_{\rm err}\) only in the coordinate
constraints.
Convexity and the residual certificate are preserved.
Containment and invariance still concern the exact ball
of radius \(R\).

After a numerical solve, verify \(\mu_z\ge\underline\mu\)
and recompute \(\tau_z,R_\sharp(z)\) independently.
Set
\(R_{\rm check}=\max\{\norm{(x^0,y^0-z)},R_\sharp(z)\}\)
and recompute the margins.
For inverse-deadline output, also check
\(\eps_i\ge\beta|x_i^0|\) using this radius and any declared
residual inflation.
Exactness of the conic formulation does not certify
floating-point output.
Rigorous certification requires outward error bounds
or exact verification.
\end{remark}

The attainable margins have an intrinsic limitation.
The next bound holds for every center in this family,
without a feasibility floor or rank assumption.

\begin{proposition}[All-center margin obstruction]
\label{prop:all-center-obstruction}
For every strictly dual-feasible \(z\), every
\(R\ge R_\sharp(z)\), and every coordinate \(i\),
\begin{equation}\label{eq:all-center-obstruction}
\lambda-|B_i^\top z|-\norm{B_i}R
\le\lambda-\frac{|B_i^\top b|}{\gamma}.
\end{equation}
If \(|B_i^\top b|\ge\gamma\lambda\), no ball in this family
has a positive margin for coordinate \(i\), at any observed state.
\end{proposition}

\begin{proof}
The endpoint in \eqref{Rsharp} gives \(R\ge\tau_z/\gamma\).
The triangle and Cauchy--Schwarz inequalities then yield
\[
|B_i^\top z|+\norm{B_i}R
\ge |B_i^\top z|+\frac{|B_i^\top(b+\gamma z)|}{\gamma}
\ge\frac{|B_i^\top b|}{\gamma}.
\]
Subtracting from \(\lambda\) proves the bound.
\end{proof}

This obstruction concerns positive margins from the specified
balls. It does not rule out identification by the dynamics,
and no converse is asserted.

\begin{proposition}[Zero-center obstruction for equal column norms]
\label{prop:obstruction}
Assume \(\norm{B_i}=c>0\) for every \(i\).
For \(z=0\), if the margin \eqref{ballmargin} is positive
for even one coordinate, then \(x=0\) is the unique minimizer
of \eqref{primal}.
Equivalently, if the problem has a nonzero minimizer,
the zero-center test cannot certify any coordinate
at any initial state.
\end{proposition}

\begin{proof}
\emph{Step 1: strict optimality margin.}
At \(z=0\), the endpoint \(t=\norm b/\gamma\)
in \eqref{Rsharp} gives
\(R\ge R_\sharp(0)\ge\norm b/\gamma\).
All coordinate margins equal \(\lambda-cR\).
If one is positive, then
\[
\frac{\norm{B^\top b}_\infty}{\gamma}
\le\frac{c\norm b}{\gamma}
\le cR<\lambda.
\]

\emph{Step 2: uniqueness.}
For every \(x\ne0\),
\[
P(x)-P(0)
=\lambda\norm x_1-\frac{\ip{B^\top b}{x}}{\gamma}
+\frac{\norm{Bx}^2}{2\gamma}
\ge\left(\lambda-
\frac{\norm{B^\top b}_\infty}{\gamma}\right)\norm x_1>0.
\]
This proves uniqueness and the contrapositive.
\end{proof}

The next example changes only the certificate center,
keeping the initial state and trajectory fixed.
The improved margin attains the all-center bound.

\begin{proposition}[An optimal center for a normalized correlated design]
\label{prop:centerexample}
Let \(m=p=2\), \(\lambda=\gamma=1\), and
\[
B=\begin{pmatrix}1&4/5\\0&3/5\end{pmatrix},
\quad b=(11/10,0),\quad
x^0=(0,1/100),\quad y^0=(-7/10,0).
\]
Both columns have norm one and are correlated.
The unique primal optimum is \(x_*=(1/10,0)\ne0\),
and the dual optimum is \(y_*=(-1,0)\).
The center \(z_{\rm base}=(-7/10,0)\) gives
\(R=2/5\), \(\eps_{2,z_{\rm base}}=1/25\), and deadline \(1/4\).
The center \(z_{\rm opt}=(-4/5,9/40)\) gives
\(R=3/8\), \(\eps_{2,z_{\rm opt}}=3/25\), and deadline \(1/12\).
The latter margin is maximal among all strictly feasible
centers in this invariant-ball family whose balls contain
the observed state.
For \(s=1/100\), the respective certified resolvent deadlines
are \(25\) and \(9\) steps.
The zero-center test fails on the same data.
\end{proposition}

\begin{proof}
\emph{Step 1: the equilibrium.}
The KKT relations hold:
\(Bx_*-b=(-1,0)=y_*\) and
\(-B^\top y_*=(1,4/5)\in\partial\norm{x_*}_1\).
Since \(B\) is invertible, the quadratic part of \(P\)
is strictly convex, so \(x_*\) is unique.

\emph{Step 2: the baseline center.}
At \(z_{\rm base}\), \(\mu_z=3/10\) and \(\tau_z=2/5\).
Lemma~\ref{lem:closed} gives \(R_\sharp=2/5\).
The observed radius is \(1/100\), so \(R=2/5\).
Since \(|B_2^\top z_{\rm base}|=14/25\), the margin is
\(1-14/25-2/5=1/25\), giving deadline \(1/4\).

\emph{Step 3: the optimal center.}
At \(z_{\rm opt}\),
\[
B^\top z_{\rm opt}=(-4/5,-101/200),\qquad
\mu_{z_{\rm opt}}=1/5,\qquad \tau_{z_{\rm opt}}=3/8.
\]
Lemma~\ref{lem:closed} gives \(R_\sharp=3/8\), while
\(\norm{(x^0,y^0-z_{\rm opt})}=\sqrt{2429}/200<3/8\).
Thus \(R=3/8\), and
\[
\eps_{2,z_{\rm opt}}=1-\frac{101}{200}-\frac38=\frac3{25},
\qquad
\frac{|x_2^0|}{\eps_{2,z_{\rm opt}}}=\frac1{12}.
\]
Since \(B_2^\top b=22/25\),
Proposition~\ref{prop:all-center-obstruction} bounds every
such margin by \(1-22/25=3/25\).
This proves optimality.
The ceilings at \(s=1/100\) give \(25\) and \(9\) steps.

\emph{Step 4: the zero center.}
For \(z=0\), \(R_\sharp(0)\ge\norm b=11/10\),
so the normalized-column margin is at most \(-1/10\).
\end{proof}

\begin{remark}[Meaning of the examples]\label{rem:center-examples}
For equal-norm designs with a nonzero minimizer,
Proposition~\ref{prop:obstruction} rules out the zero-center test.
Proposition~\ref{prop:centerexample} shows that another center
can overcome this failure, even with correlated columns.
For coordinate \(2\), maximizing the margin also minimizes
the certified deadline within this ball family.
The value \(1/12\) is an upper bound on the hitting time;
the coordinate may reach zero earlier.
The examples show that useful certificates can exist,
but do not measure how often they occur on high-dimensional
measured data.
\end{remark}
\section{Numerical implementation and applications}\label{sec:numerics}
We compare identification deadlines with observed zeros in
numerical resolvent trajectories.
The experiments test center selection and the conservatism
of the bounds.
A weak-form oscillator regression then examines the permanence
criterion for a fitted zero.

\subsection{Semismooth Newton solution of the resolvent equation}

Semismooth Newton methods for \(\ell^1\)-regularized least squares
are established \cite[Section~3.2]{LiSunToh}.
For the full resolvent equation, simultaneous thresholds require
distinguishing the Clarke Jacobian from the matrix family used
by the Newton solver.

Eliminating the primal variable from \eqref{resolvent} gives
the equation for \(Y=y^{k+1}\):
\begin{equation}\label{F}
F(Y)=(1+s\gamma)Y-y^k
-s\big[B\soft_{s\lambda}(x^k-sB^\top Y)-b\big]=0.
\end{equation}
Away from the scalar thresholds, the ordinary Jacobian is
\begin{equation}\label{Jac}
JF(Y)=(1+s\gamma)I_m+
s^2B\operatorname{diag}(d_i)B^\top,\quad
d_i=\mathbf1_{\{|x_i^k-sB_i^\top Y|>s\lambda\}}.
\end{equation}

At simultaneous thresholds, independent choices
\(d_i\in[0,1]\) need not give a member of
\(\partial_C F(Y)\).
Put \(\mathcal J=\{i:B_i\ne0\}\), and let
\(\mathscr A(Y)\) contain the binary masks on \(\mathcal J\)
attained arbitrarily close to \(Y\) away from all nonconstant
threshold hyperplanes.
Such points exist because only finitely many proper
hyperplanes are excluded.
Set \(d_i=0\) for \(i\notin\mathcal J\).
Zero-column masks have no effect, even when their constant
arguments remain at threshold.
Since \(F\) is continuous and piecewise affine,
\begin{equation}\label{Clarke}
\partial_C F(Y)=
\co\left\{(1+s\gamma)I_m+
s^2B\operatorname{diag}(d)B^\top:
d\in\mathscr A(Y)\right\}.
\end{equation}

For Newton iteration, use
\(J(Y)=(1+s\gamma)I+s^2B\operatorname{diag}(d_i(Y))B^\top\),
with \(d_i=1\) strictly outside the thresholds,
\(d_i=0\) strictly inside, and \(d_i\in[0,1]\) at equality.
This family can strictly contain \(\partial_C F(Y)\),
but it satisfies
\[
\norm{F(Y+h)-F(Y)-J(Y+h)h}=o(\norm h).
\]
For a fixed \(Y\), the remainder is zero for sufficiently
small \(h\): each scalar increment lies on one affine branch
or starts at its endpoint.
An argument at threshold at both endpoints has zero increment.
The script chooses \(d_i=0\) at equality.

Positive definiteness ensures an invertible Newton matrix,
but not descent of the residual norm.
We therefore use the potential
\begin{equation}\label{potential}
\Phi(Y)=\frac{1+s\gamma}{2}\norm Y^2
-\ip{y^k-sb}{Y}
+\frac12\norm{\soft_{s\lambda}(x^k-sB^\top Y)}^2.
\end{equation}
The scalar function
\(w\mapsto\soft_{s\lambda}(w)^2/2\)
has derivative \(\soft_{s\lambda}(w)\), so
\(\nabla\Phi(Y)=F(Y)\).
The potential is \((1+s\gamma)\)-strongly convex
and has a globally Lipschitz gradient.
For any
\(J=(1+s\gamma)I+s^2B\operatorname{diag}(d)B^\top\),
\(0\le d_i\le1\), the direction \(p=-J^{-1}F(Y)\) satisfies
\[
\ip{\nabla\Phi(Y)}p=-F(Y)^\top J^{-1}F(Y)<0
\quad\text{if }F(Y)\ne0.
\]
Thus Armijo backtracking on \(\Phi\) accepts a sufficiently
short step.
For a fixed resolvent problem, the bounds
\((1+s\gamma)I\preceq J\preceq
(1+s\gamma+s^2\norm B^2)I\)
also give uniform descent constants.

\begin{proposition}[Simultaneous thresholds are coupled]\label{prop:jac}
Let \(s=\lambda=\gamma=1\),
\(B=(1,1)\), \(x^k=(1,-1)\), and \(y^k=b=0\).
Then \(F(Y)=3Y\) near \(Y=0\), so
\(\partial_C F(0)=\{3\}\).
The independent choice \(d_1=d_2=0\) gives \(2\),
which is not in \(\partial_C F(0)\).
\end{proposition}

\begin{proof}
For small positive \(Y\), only the second soft-thresholded
component is nonzero and equals \(-Y\).
For small negative \(Y\), only the first is nonzero
and equals \(-Y\).
Their sum is therefore \(-Y\) on both sides, giving
\(F(Y)=2Y-(-Y)=3Y\).
Hence \(\partial_C F(0)=\{3\}\).
\end{proof}

\begin{proposition}[Residual-norm backtracking can fail]\label{prop:armijo}
Let \(s=\lambda=\gamma=1\),
\[
B=\begin{pmatrix}2&6\\0&-6\end{pmatrix},
\quad x^k=(2,-1),\quad y^k=0,\quad b=(4,1).
\]
At \(Y=0\), the strict-active mask gives
\(d=(1,0)\), \(F(0)=(2,1)\), and
\(J=\operatorname{diag}(6,2)\).
For \(p=-J^{-1}F(0)=(-1/3,-1/2)\),
\(\norm{F(\alpha p)}>\norm{F(0)}\) for every \(\alpha>0\).
\end{proposition}

\begin{proof}
Along the ray, the threshold arguments are
\(2+2\alpha/3>1\) and \(-1-\alpha<-1\).
Substitution gives \(F(\alpha p)=(2+4\alpha,1-7\alpha)\), so
\[
\norm{F(\alpha p)}^2=5+2\alpha+65\alpha^2
>5=\norm{F(0)}^2.
\]
The same direction decreases \(\Phi\) locally, since
\(\ip{\nabla\Phi(0)}p=-7/6<0\).
\end{proof}

\subsection{Certificate diagnostics}

The script \texttt{reproduce\_certificates.py} solves
\eqref{F} using the strict-active Newton selection and
Armijo backtracking on \(\Phi\).
The primal update uses soft thresholding with no additional
coordinate deletion.
The stopping tolerance is \(10^{-11}\) for the Euclidean norm
of the computed dual-equation residual.
Gap diagnostics use Theorem~\ref{thm:gapcert}.
Restricted-rank tests require \(B_{I^c}\) to have no more columns
than rows and least singular value greater than \(10^{-10}\).

All numerical states and margins below are floating-point
diagnostics.
Small residuals alone do not certify the trajectory:
Theorem~\ref{thm:inexact} requires a bound on the full future
residual sum, and rigorous margins require outward enclosures.
The experiments assess certificate timing and make no
solver-speed claim.
Stress tests use fixed reference-inactive sets rather than
an automatic candidate-set selection procedure.
A \emph{false nonzero} is a coordinate in such a set whose
magnitude exceeds the declared numerical-zero threshold.

\paragraph*{Normalized correlated analytic instance.}
For Proposition~\ref{prop:centerexample}, take \(s=0.01\).
The centers \(z_{\rm base}=(-0.7,0)\) and
\(z_{\rm opt}=(-0.8,9/40)\) give margins \(0.04\) and \(0.12\).
Their continuous deadlines are \(1/4\) and \(1/12\),
and their discrete deadlines are \(25\) and \(9\) steps.
Both concern the same initial state and trajectory.
Coordinate \(2\) first becomes zero numerically at step \(3\)
and remains zero through step \(60\).
The largest recorded algebraic residual is below
\(1.5\cdot10^{-16}\).

\begin{figure}[!htbp]
\centering
\includegraphics[width=\linewidth]{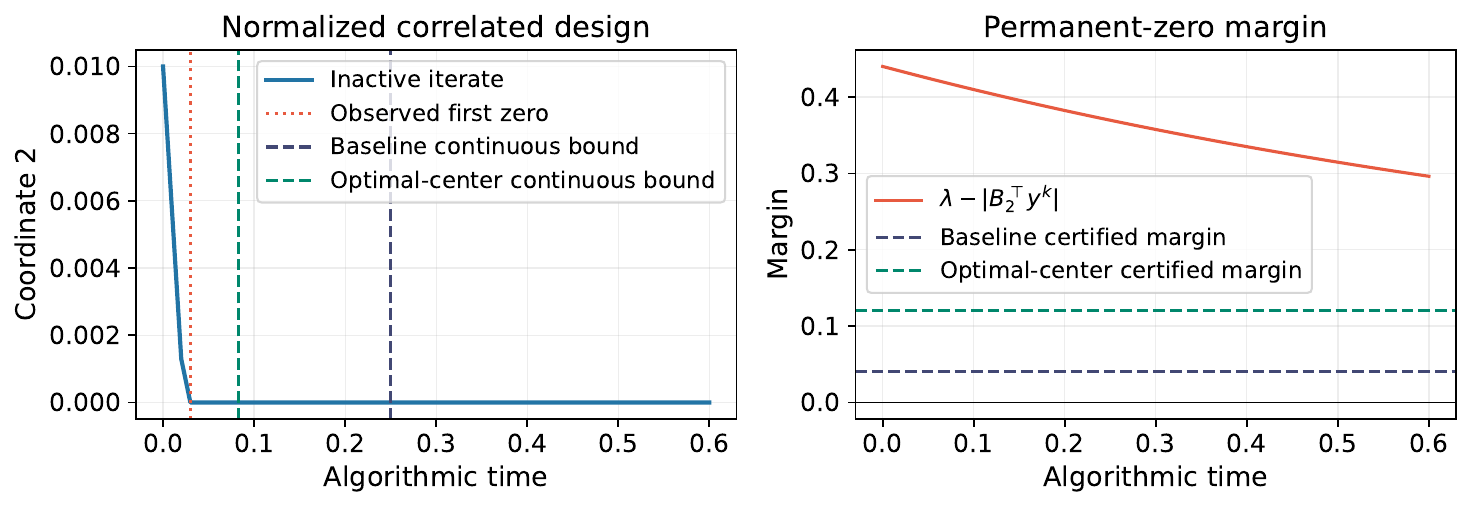}
\caption{Normalized correlated design, \(s=0.01\).
Left: observed zero and certified deadlines for the two centers.
Right: realized dual margin and initial uniform margins.}
\label{fig:normalized}
\end{figure}
\FloatBarrier

\paragraph*{Finite center-design validation.}
For the same state, \(I=\{2\}\), and
\(\underline\mu=0.2\), solve \eqref{eq:center-design}
using Proposition~\ref{prop:finite-center-sdp}.
CVXPY with Clarabel returns margin \(0.12000000\),
matching the exact bound \(3/25\).
The returned center is approximately
\((-0.518866,0.435851)\), with radius \(0.726418\).
Independent checks of \(\mu_z,\tau_z,R_\sharp(z)\)
and state containment reproduce the margin to the stated precision.
The coefficient-equation residual is below \(2\cdot10^{-13}\),
and the maximum computed conic violation is below
\(1.1\cdot10^{-11}\).
These are solver diagnostics; analytic optimality follows from
the obstruction and its attaining example.
Maximizing centers need not be unique, so the comparison concerns
margins and deadlines.

\paragraph*{Small-step consistency of the eventual bound.}
Fix the same observed state, \(z=(-0.7,0)\), \(I=\{2\}\),
and \(\delta=1/5\).
The optimizer margin is \(\eta_I\simeq0.0368871\),
and \(B_{I^c}\) has singular value \(1\).
The enclosure gives \(d\simeq0.577289\) and
\(\rho\simeq0.00218084\).
Table~\ref{tab:small-step} evaluates \eqref{discdeadline}:
\(sK_s\) approaches the continuous bound \(3445.294655\),
as in Remark~\ref{rem:small}.
These are deadline bounds, not simulated hitting times.
They remain much larger than the invariant-ball deadlines
at the same state.

\begin{table}[!htbp]
\centering
\caption{Eventual bounds for fixed state and parameters.
The limiting value of \(sK_s\) is \(3445.294655\).}
\label{tab:small-step}
\begin{tabular}{rrrr}
\hline
$s$ & $N_s$ & $K_s$ & $sK_s$\\
\hline
0.5 & 6893 & 6895 & 3447.500\\
0.1 & 34447 & 34457 & 3445.700\\
0.05 & 68890 & 68910 & 3445.500\\
0.01 & 344434 & 344534 & 3445.340\\
0.005 & 688863 & 689063 & 3445.315\\
\hline
\end{tabular}
\end{table}
\FloatBarrier

\paragraph*{Unequal-norm comparison.}
Take \(B=(1,0.3)\), \(b=2\), \(\lambda=\gamma=1\),
\(x^0=(2,1)\), and \(y^0=z=0\).
Then \(R=\sqrt5\), \(R_\sharp=2\), and the coordinate-\(2\)
margin is \(1-0.3\sqrt5\simeq0.329180\).
The continuous deadline is \(3.037855\).
At \(s=0.005\), the discrete deadline is \(608\) steps,
while the first numerical zero occurs at step \(205\).
Unequal column norms permit this zero-center success
without contradicting Proposition~\ref{prop:obstruction}.
The comparison appears in the supplementary figure.

\paragraph*{Underdetermined stress test.}
Using NumPy seed \(13\), generate a \(40\times60\) matrix
with independent standard normal entries divided by \(\sqrt{40}\).
Prescribe nonzero coordinates at zero-based positions
\(1,8,23,35,51\), with values \((1.3,-1.5,1,1.2,-1)\).
For \(\gamma=1\), \(\lambda=0.4\), and this support \(S\), set
\[
y_*=-\lambda B_S(B_S^\top B_S)^{-1}
\operatorname{sign}(x_{*,S}),\qquad
b=Bx_*-y_*.
\]
This gives a KKT pair with smallest unused-coordinate
strict margin \(0.03855565\) and least singular value
\(0.68901180\) for \(B_S\).
Using the same generator, draw \(x^0\) with independent normal
components of standard deviation \(0.15\).
Set \(y^0=0\) and \(s=0.05\).
Test the optimizer and orbit margins on all \(55\) excluded
coordinates every fifth iterate.

The last false nonzero occurs at optimization time \(4.75\).
The first sampled positive optimizer screen occurs at \(17.00\),
and the first positive uniform orbit margin at \(20.00\).
The maximum implicit residual through \(600\) steps is below
\(2.3\cdot10^{-16}\).
The constructed KKT pair supplies the reference-inactive set,
and the restricted-rank assumption holds for these \(55\) exclusions.
Both sufficient tests become positive after the observed zeros.

\begin{figure}[!htbp]
\centering
\includegraphics[width=\linewidth]{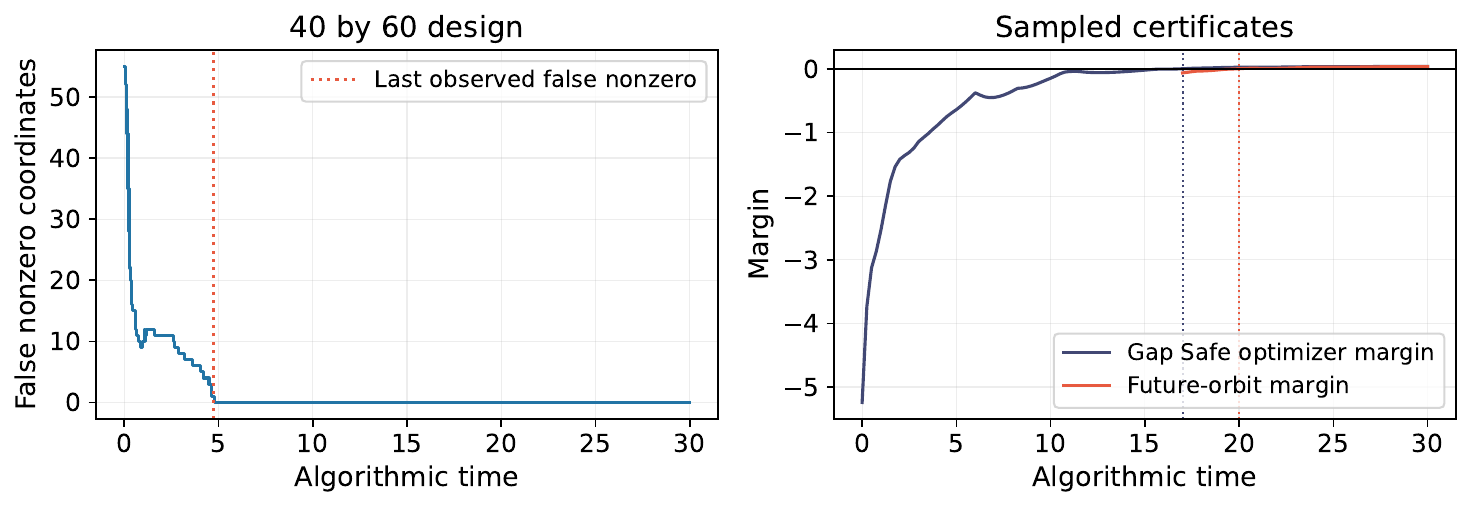}
\caption{Controlled \(40\times60\) design.
The false-nonzero count reaches zero before either sampled
whole-set certificate becomes positive.
The orbit margin is plotted only after the optimizer screen passes.}
\label{fig:underdetermined}
\end{figure}
\FloatBarrier

\paragraph*{Measured-data stress test.}
The optional script reads the public Wisconsin Diagnostic
table through scikit-learn \cite{Wisconsin}
for a sparse feature-selection test, not a clinical model.
Use a stratified \(75\%\) training split with seed \(2026\),
based on the original labels: benign \(=1\), malignant \(=0\).
After the split, map the fitting labels to malignant \(+1\)
and benign \(-1\).
Standardize features using training means and standard deviations,
and center the mapped training labels.
Set
\(B=X_{\rm train}/\sqrt{426}\),
\(b=(\ell-\bar\ell)/\sqrt{426}\),
\(\gamma=1\), and
\(\lambda=0.04\norm{B^\top b}_\infty\).
The columns have norm one up to rounding, and
\(\lambda=0.03124543\) for this split.
An independent high-accuracy reference Lasso fit gives
\(22\) numerically zero coefficients, defining the
reference-inactive set.

With \(x^0=y^0=0\) and \(s=1\), solve \(1000\) resolvent
equations numerically.
The last false nonzero occurs at step \(67\).
Testing every fifth step, the first all-\(22\) positive
optimizer screen occurs at step \(150\), and the first
positive future-orbit margin at step \(220\).
The largest recorded implicit residual is \(1.0\cdot10^{-11}\).
Again, the whole-set tests become positive after the observed zeros.

For \(20\) of the \(22\) reference-inactive coordinates,
the reproduced data satisfy \(|B_i^\top b|\ge\lambda\).
Proposition~\ref{prop:all-center-obstruction} therefore excludes
positive invariant-ball margins for every center and observed state.
The other two coordinates remain undecided by this condition.
The supplementary output also includes \(101\) feasible values
in the scan \(z=-\alpha b\), with best whole-set margin
approximately \(-7.47\).

\begin{figure}[!htbp]
\centering
\includegraphics[width=\linewidth]{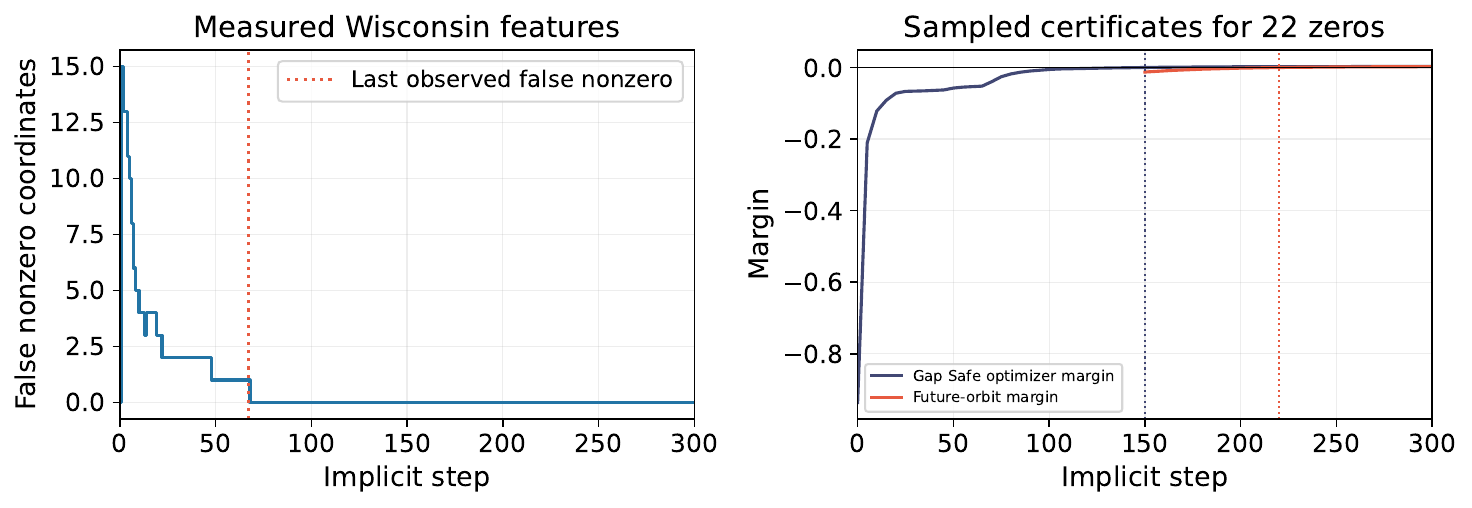}
\caption{Wisconsin training feature table:
false-nonzero count and sampled optimizer and future-orbit margins.
The forward margin is plotted only when the optimizer screen is positive.}
\label{fig:wisconsin}
\end{figure}
\FloatBarrier

\subsection{Weak-form regression for an oscillator}

We fit a sparse oscillator model from simulated noisy
displacement and velocity measurements.
A weak formulation avoids differentiating these samples.
The experiment examines a fitted zero and the criterion
for its subsequent permanence.
Physical time is distinct from optimization time or resolvent index.

Consider the driven oscillator
\begin{equation}\label{osc}
\dot q=v,\qquad
\dot v=c_1q+c_2v+c_3q^2+c_4q^3
+c_5v|v|+c_6u,
\end{equation}
with dictionary
\(\psi(q,v,u)=(q,v,q^2,q^3,v|v|,u)\).
Take compactly supported
\(\phi_j\in W^{2,1}([a_j,b_j])\) with
\(\phi_j=\dot\phi_j=0\) at both endpoints.
Assume \(q\in W^{2,1}\) on each interval and that the
right side of \eqref{osc} is integrable.
Two integrations by parts give
\begin{equation}\label{weak}
\sum_{i=1}^6c_i\int_{a_j}^{b_j}
\phi_j(t)\psi_i(q(t),v(t),u(t))\,dt
=\int_{a_j}^{b_j}\phi_j''(t)q(t)\,dt.
\end{equation}
Put \(t_j=(a_j+b_j)/2\), \(h_j=(b_j-a_j)/2>0\), and
\(s_j(t)=(t-t_j)/h_j\).
The windows \(\phi_j(t)=(1-s_j(t)^2)^3\) for
\(|s_j(t)|\le1\), extended by zero, are \(C^2\),
belong to \(W^{2,1}\), and satisfy the endpoint conditions.

Identity~\eqref{weak} holds for the underlying trajectory.
Quadrature using noisy measurements of both \(q\) and \(v\)
produces \(B^{\rm raw}\) and \(b\).
For \(n_i=\norm{B_i^{\rm raw}}\ne0\), set
\(B_i=B_i^{\rm raw}/n_i\) and \(x_i=n_i c_i\).
The objective \eqref{primal} then defines a regularized
estimator, with fitted physical coefficients \(c_i=x_i/n_i\).
Under noise or positive regularization, fitted zeros do not
certify the physical support of \eqref{osc}.
Weak sparse identification is established methodology
\cite{Messenger,Brunton}.

\paragraph*{Noisy weak-form oscillator.}
The script integrates \eqref{osc} over physical time \([0,25]\)
from \((q,v)=(0.1,0.2)\), with true coefficients
\((-1.4,-0.35,0,-0.8,0,1)\) and drive
\(u(t)=\sin(1.7t)+0.5\sin(2.9t+0.4)\).
Samples are taken every \(0.005\) units.
Independent normal noise of standard deviation \(0.002\)
is added to \(q,v\) using seed \(2026\).
Use \(86\) polynomial windows \((1-s_j^2)^3\)
of half-width \(0.45\), centered at
\(0.5+0.28j\), \(j=0,\ldots,85\).
The regression has normalized columns,
\(\lambda=0.025\), and \(\gamma=1\).

L-BFGS-B optimization on nonnegative split coefficients
gives fitted physical coefficients approximately
\[
(-1.40446,-0.31190,0,-0.81006,-0.00726,0.94811).
\]
Only the \(q^2\) coefficient is numerically zero,
with approximate dual strict margin \(0.02150\).
The physical \(v|v|\) coefficient is also zero,
but its fitted value is nonzero.

Let \(x^{\rm ref}\) be the reference fit in normalized coordinates.
Initialize \(x^0=x^{\rm ref}\), except that the \(q^2\)
coordinate is set to \(0.0002\), and take
\(y^0=(Bx^{\rm ref}-b)/\gamma\).
For \(s=0.01\), the initial computed optimizer margin is
\(0.01809\), while the future-orbit margin is \(-0.01145\).
After the first numerical resolvent step, the coefficient
is zero and the future-orbit margin is \(0.01742\).
Thus the permanence test becomes positive after the observed zero.
This deliberately near-optimal start examines that transition;
it does not assess routine initialization.
The first algebraic residual is below \(2.0\cdot10^{-12}\),
and the reference fit's dual-feasibility violation is below
\(2\cdot10^{-9}\).

\begin{figure}[!htbp]
\centering
\includegraphics[width=\linewidth]{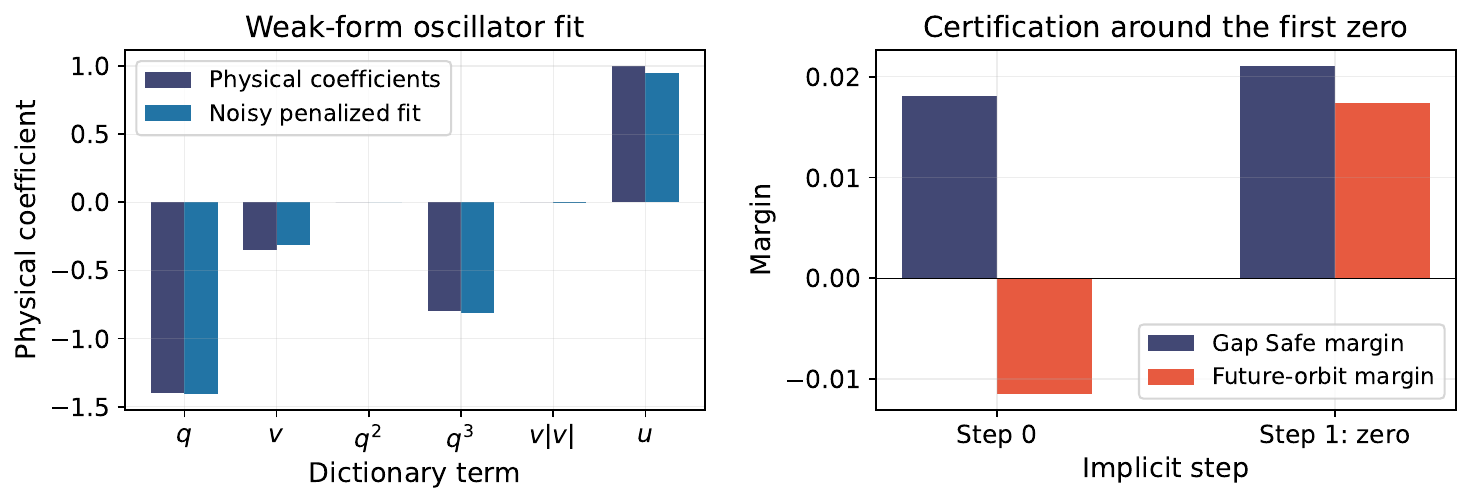}
\caption{Noisy weak-form oscillator regression.
Left: physical and fitted coefficients.
Right: computed optimizer and future-orbit margins
before and after the first resolvent step.}
\label{fig:duffing}
\end{figure}
\FloatBarrier

\paragraph*{Reproduction protocol.}
The companion scripts generate the four displayed vector figures,
the supplementary unequal-norm figure, Table~\ref{tab:small-step},
and \texttt{certificate\_results.json}.
They specify the parameters, seeds, preprocessing, and solver settings.
The JSON records full-precision constants, inactive indices,
package versions, center-program matrices, and constraint checks.
The requirements file specifies the tested environment.

Numerical-zero thresholds are \(10^{-10}\) for controlled
trajectories and \(10^{-9}\) for Wisconsin iterates.
Reference-fit thresholds are \(10^{-8}\) for Wisconsin
and \(10^{-7}\) for the oscillator.
Finite-run persistence is distinct from theorem-based permanence.
Stress-test margins are sampled every five iterates.
Tiny negative computed gaps are clipped to zero for diagnostics;
this is not outward rounding.
Dependency versions may change last digits or sampled onsets.
The analytic margins in Proposition~\ref{prop:centerexample}
are independent of software.
\section{Conclusions}\label{sec:conclusions}
Invariant balls give computable deadlines for permanent zeros
in finite-dimensional primal--dual flows and exact full resolvents,
without optimizer knowledge or rank assumptions.
Within this center family, we determine the smallest invariant
radius for the continuous flow and give a convex center-selection
problem with an exact finite conic formulation for a prescribed
positive feasibility margin.
The obstruction and optimal-center example describe the limits
of this construction.
The composite Bregman estimate proves rank-free convergence
and finite identification under strict dual margins.
With restricted injectivity, complementary gap-based deadlines
have discrete bounds that, multiplied by the step size, converge
to the continuous bound for fixed data, observed state, and
certificate parameters.
These sufficient bounds can be conservative.
Inexact guarantees require exact primal thresholding and a known
bound on the sum of all future dual residual norms.
In the oscillator application, fitted support may differ
from physical support.





\bibliographystyle{tfnlm}
\bibliography{references}

@article{Bregman,
  author  = {Bregman, L. M.},
  title   = {The relaxation method of finding the common point of convex sets and its application to the solution of problems in convex programming},
  journal = {USSR Comput. Math. Math. Phys.},
  volume  = {7},
  year    = {1967},
  pages   = {200--217}
}

@book{Brezis,
  author    = {Br{\'e}zis, H.},
  title     = {Op{\'e}rateurs maximaux monotones et semi-groupes de contractions dans les espaces de {Hilbert}},
  publisher = {North-Holland},
  address   = {Amsterdam},
  year      = {1973}
}

@article{Bonnefoy,
  author  = {Bonnefoy, A. and Emiya, V. and Ralaivola, L. and Gribonval, R.},
  title   = {Dynamic screening: accelerating first-order algorithms for the {Lasso} and group-{Lasso}},
  journal = {IEEE Trans. Signal Process.},
  volume  = {63},
  year    = {2015},
  pages   = {5121--5132},
  doi     = {10.1109/TSP.2015.2447503}
}

@article{Brunton,
  author  = {Brunton, S. L. and Proctor, J. L. and Kutz, J. N.},
  title   = {Discovering governing equations from data by sparse identification of nonlinear dynamical systems},
  journal = {Proc. Natl. Acad. Sci. USA},
  volume  = {113},
  year    = {2016},
  pages   = {3932--3937}
}

@article{Dantas,
  author  = {Dantas, C. F. and Soubies, E. and F{\'e}votte, C.},
  title   = {Expanding boundaries of {Gap Safe} screening},
  journal = {J. Mach. Learn. Res.},
  volume  = {22},
  number  = {236},
  year    = {2021},
  pages   = {1--57}
}

@misc{Diaz,
  author       = {D{\'i}az, M. and Izquierdo Lehmann, P. and Lu, H. and Yang, J.},
  title        = {Active set identification and rapid convergence for degenerate primal-dual problems},
  howpublished = {arXiv:2602.10436},
  year         = {2026},
  note         = {Theorems 4.2--4.3}
}

@article{Garrigos,
  author  = {Garrigos, G. and Rosasco, L. and Villa, S.},
  title   = {Thresholding gradient methods in {Hilbert} spaces: support identification and linear convergence},
  journal = {ESAIM Control Optim. Calc. Var.},
  volume  = {26},
  year    = {2020},
  note    = {Article 28}
}

@article{Liang,
  author  = {Liang, J. and Fadili, J. and Peyr{\'e}, G.},
  title   = {Local linear convergence analysis of primal--dual splitting methods},
  journal = {Optimization},
  volume  = {67},
  year    = {2018},
  pages   = {821--853},
  note    = {Theorems 3.3 and 3.11}
}

@article{LewisTian,
  author  = {Lewis, A. S. and Tian, T.},
  title   = {Identifiability, the {KL} property in metric spaces, and subgradient curves},
  journal = {Found. Comput. Math.},
  volume  = {25},
  year    = {2025},
  pages   = {905--942}
}

@misc{LuQu,
  author       = {Lu, M. and Qu, Z.},
  title        = {An adaptive proximal point algorithm framework and application to large-scale optimization},
  howpublished = {arXiv:2008.08784v2},
  year         = {2021}
}

@article{Messenger,
  author  = {Messenger, D. A. and Bortz, D. M.},
  title   = {Weak {SINDy}: {Galerkin}-based data-driven model selection},
  journal = {Multiscale Model. Simul.},
  volume  = {19},
  year    = {2021},
  pages   = {1474--1497}
}

@article{Ndiaye,
  author  = {Ndiaye, E. and Fercoq, O. and Gramfort, A. and Salmon, J.},
  title   = {{Gap Safe} screening rules for sparsity enforcing penalties},
  journal = {J. Mach. Learn. Res.},
  volume  = {18},
  number  = {128},
  year    = {2017},
  pages   = {1--33},
  note    = {Theorem 6 and Proposition 14}
}

@article{Rockafellar,
  author  = {Rockafellar, R. T.},
  title   = {Monotone operators and the proximal point algorithm},
  journal = {SIAM J. Control Optim.},
  volume  = {14},
  year    = {1976},
  pages   = {877--898}
}

@article{SaoudComp,
  author  = {Saoud, H.},
  title   = {Composite {Lyapunov} criteria for stability and convergence with applications to optimization dynamics},
  journal = {Mathematics},
  volume  = {13},
  year    = {2025},
  pages   = {3859}
}

@article{SaoudGeo,
  author  = {Saoud, H. and Th{\'e}ra, M. and Dao, M. N.},
  title   = {Geometric stability analysis for differential inclusions governed by maximally monotone operators},
  journal = {Set-Valued Var. Anal.},
  volume  = {34},
  year    = {2026},
  note    = {Article 22},
  doi     = {10.1007/s11228-026-00810-9}
}

@article{SaoudLoc,
  author  = {Dao, M. N. and Saoud, H. and Th{\'e}ra, M.},
  title   = {Locating theorems of differential inclusions governed by maximally monotone operators},
  journal = {SIAM J. Optim.},
  volume  = {33},
  year    = {2023},
  pages   = {2703--2720}
}

@misc{Wisconsin,
  author       = {Street, W. N. and Wolberg, W. H. and Mangasarian, O. L.},
  title        = {{Breast Cancer Wisconsin (Diagnostic)}},
  howpublished = {UCI Machine Learning Repository},
  year         = {1995},
  url          = {https://archive.ics.uci.edu/dataset/17/breast+cancer+wisconsin+diagnostic}
}

@article{Burger,
  author  = {Burger, M. and Osher, S.},
  title   = {Convergence rates of convex variational regularization},
  journal = {Inverse Problems},
  volume  = {20},
  number  = {5},
  year    = {2004},
  pages   = {1411--1421},
  doi     = {10.1088/0266-5611/20/5/005}
}

@article{ChenWei,
  author  = {Chen, L. and Wei, J.},
  title   = {Transformed primal--dual methods for nonlinear saddle point systems},
  journal = {Journal of Numerical Mathematics},
  volume  = {31},
  number  = {4},
  year    = {2023},
  pages   = {281--311},
  doi     = {10.1515/jnma-2022-0056}
}

@incollection{Parrilo,
  author    = {Parrilo, P. A.},
  title     = {Polynomial optimization, sums of squares, and applications},
  editor    = {Blekherman, G. and Parrilo, P. A. and Thomas, R. R.},
  booktitle = {Semidefinite Optimization and Convex Algebraic Geometry},
  series    = {MOS--SIAM Series on Optimization},
  volume    = {13},
  publisher = {SIAM},
  address   = {Philadelphia},
  year      = {2013},
  pages     = {47--157},
  doi       = {10.1137/1.9781611972290.ch3}
}

@article{LiSunToh,
  author  = {Li, X. and Sun, D. and Toh, K.-C.},
  title   = {A highly efficient semismooth {Newton} augmented {Lagrangian} method for solving {Lasso} problems},
  journal = {SIAM Journal on Optimization},
  volume  = {28},
  year    = {2018},
  pages   = {433--458},
  doi     = {10.1137/16M1097572}
}

@article{Tibshirani,
  author  = {Tibshirani, R.},
  title   = {Regression shrinkage and selection via the {Lasso}},
  journal = {J. R. Statist. Soc. Ser. B (Methodol.)},
  volume  = {58},
  year    = {1996},
  pages   = {267--288},
  doi     = {10.1111/j.2517-6161.1996.tb02080.x}
}

\end{document}